\documentclass{amsart}
\usepackage{amssymb}
\usepackage{amsmath, amscd}
\usepackage{amsthm}
\usepackage{mathrsfs}
\usepackage{perpage}
\usepackage[all,cmtip]{xy}
\usepackage{setspace}
\usepackage{amsbsy}
\usepackage[foot]{amsaddr}
\usepackage{url}
\usepackage[T1]{fontenc}
\usepackage{verbatim}
\usepackage{mathrsfs}
\usepackage{ifthen}
\usepackage{blkarray}
\usepackage{multirow}
\usepackage{enumerate}
\usepackage{tikz-cd}
\usepackage{mathtools}
\usepackage{etoolbox}
\usetikzlibrary{positioning}

\title[Automorphic bundles on Zip-schemes]{Automorphic vector bundles with global sections on $\GZip^\Zcal$-schemes}

\newtheorem{theorem}[subsubsection]{Theorem}
\newtheorem{lemma}[subsubsection]{Lemma}
\newtheorem{conjecture}[subsubsection]{Conjecture}

\newtheorem{corollary}[subsubsection]{Corollary}
\newtheorem{definition}[subsubsection]{Definition}
\newtheorem{question}[subsubsection]{Question}

\newtheorem{proposition}[subsubsection]{Proposition}

\newtheorem{qAlph}{Question}

\newtheorem{thmAlph}[qAlph]{Theorem}

\newtheorem{conjAlph}[qAlph]{Conjecture}

\theoremstyle{remark}

\newtheorem{rmk}[subsubsection]{Remark}

\numberwithin{equation}{subsection}

\AtBeginEnvironment{equation}{\setcounter{equation}{\value{subsubsection}}}{}{}{}
\AtEndEnvironment{equation}{\stepcounter{subsubsection}}{}{}{}

\AtBeginEnvironment{subequation}{\setcounter{subequation}{\value{subsubsections}}}{}{}{}
\AtEndEnvironment{subequation}{\stepcounter{subsubsection}}{}{}{}
 
\AtEndDocument{\bigskip{\footnotesize%
\noindent (Wushi Goldring)
  \textsc{Dept. of Mathematics, Stockholm U., Stockholm SE-10691, Sweden} \par  
 \noindent \texttt{wushijig@gmail.com} \par
  \addvspace{\medskipamount}
\noindent (Jean-Stefan Koskivirta) \textsc{Department of Mathematics, South Kensington Campus, Imperial College London, London SW7 2AZ, UK} \par 
\noindent \texttt{jeanstefan.koskivirta@gmail.com}
}}

\newcommand{\GZip}{\mathop{\text{$G$-{\tt Zip}}}\nolimits}

\newcommand{\GtildeZip}{\mathop{\text{$\tilde{G}$-{\tt Zip}}}\nolimits}

\newcommand{\GLnZip}{\mathop{\text{$GL(n)$-{\tt Zip}}}\nolimits}

\newcommand{\GF}{\mathop{\text{$G$-{\tt ZipFlag}}}\nolimits}

\DeclareMathOperator{\Gal}{Gal}

\DeclareMathOperator{\Sbt}{Sbt}

\newskip\procskipamount
\newskip\interskipamount
\newskip\refskipamount
\newcommand{\procskip}{\vskip\procskipamount}
\newcommand{\interskip}{\vskip\interskipamount}
\newcommand{\refskip}{\vskip\refskipamount}

\newcommand{\procbreak}{\par
   \ifdim\lastskip<\procskipamount\removelastskip
   \penalty-100
   \procskip\fi
   \noindent\ignorespaces}

\newcommand{\titlebreak}{\par%
\ifdim\lastskip<\interskipamount\removelastskip%
\penalty10000%
\interskip\fi%
\noindent}%

\newcommand{\interbreak}{\par%
\ifdim\lastskip<\interskipamount\removelastskip%
\penalty-100%
\interskip\fi%
\noindent\ignorespaces}%

\newcommand{\refbreak}{\par%
\ifdim\lastskip<\refskipamount\removelastskip%
\penalty-100%
\refskip\fi%
\noindent\ignorespaces}%

\newcounter{listcounter}
\newcounter{deflistcounter}
\newcounter{equivcounter}

\newskip{\itemsepamount}
\newskip{\topsepamount}
\newenvironment{assertionlist}{%
  \begin{list}
    {\upshape (\arabic{listcounter})}
    {\setlength{\leftmargin}{18pt}
     \setlength{\rightmargin}{0pt}
     \setlength{\itemindent}{0pt}
     \setlength{\labelsep}{5pt}
     \setlength{\labelwidth}{13pt}
     \setlength{\listparindent}{\parindent}
     \setlength{\parsep}{0pt}
     \setlength{\itemsep}{\itemsepamount}
     \setlength{\topsep}{\topsepamount}
     \usecounter{listcounter}}}
  {\end{list}}

\newenvironment{definitionlist}{%
  \begin{list}
    {\upshape (\alph{deflistcounter})}
    {\setlength{\leftmargin}{18pt}
     \setlength{\rightmargin}{0pt}
     \setlength{\itemindent}{0pt}
     \setlength{\labelsep}{5pt}
     \setlength{\labelwidth}{13pt}
     \setlength{\listparindent}{\parindent}
     \setlength{\parsep}{0pt}
     \setlength{\itemsep}{\itemsepamount}
     \setlength{\topsep}{\topsepamount}
     \usecounter{deflistcounter}}}
  {\end{list}}

\newenvironment{equivlist}{%
  \begin{list}
    {\upshape (\roman{equivcounter})}
    {\setlength{\leftmargin}{18pt}
     \setlength{\rightmargin}{0pt}
     \setlength{\itemindent}{0pt}
     \setlength{\labelsep}{5pt}
     \setlength{\labelwidth}{13pt}
     \setlength{\listparindent}{\parindent}
     \setlength{\parsep}{0pt}
     \setlength{\itemsep}{\itemsepamount}
     \setlength{\topsep}{\topsepamount}
     \usecounter{equivcounter}}}
  {\end{list}}

% ...
\newcommand{\Acal}{{\mathcal A}}
\newcommand{\Bcal}{{\mathcal B}}

\newcommand{\Mcal}{{\mathcal M}}
\newcommand{\Ncal}{{\mathcal N}}
\newcommand{\Ocal}{{\mathcal O}}
\newcommand{\Pcal}{{\mathcal P}}

\newcommand{\Xcal}{{\mathcal X}}
\newcommand{\Ycal}{{\mathcal Y}}
\newcommand{\Zcal}{{\mathcal Z}}

\renewcommand{\AA}{\mathbf{A}}

\newcommand{\CC}{\mathbf{C}}

\newcommand{\FF}{\mathbf{F}}
\newcommand{\GG}{\mathbf{G}}

\newcommand{\NN}{\mathbf{N}}

\newcommand{\QQ}{\mathbf{Q}}
\newcommand{\RR}{\mathbf{R}}

\newcommand{\XX}{\mathbf{X}}

\newcommand{\ZZ}{\mathbf{Z}}

\newcommand{\ee}{\mathbf{e}}

\DeclareMathOperator{\Pic}{Pic}

\newcommand{\Lscr}{{\mathscr L}}

\newcommand{\Vscr}{{\mathscr V}}

\newcommand{\cent}{{\rm Cent}}

\newcommand{\po}{\mathbf P^1}

\newcommand{\fp}{\mathbf F_p}

\newcommand{\fpbar}{\overline{\mathbf F}_p}

\newcommand{\ei}{\mathbf e _i}
\newcommand{\ej}{\mathbf e_j}

\newcommand{\gx}{(\mathbf G, \mathbf X)}

\newcommand{\ad}{{\rm ad}}

\newcommand{\Th}{{\rm Th.}}
\newcommand{\Ths}{{\rm Ths.}}
\newcommand{\Rmk}{{\rm Rmk.}}

\newcommand{\Cor}{{\rm Cor.}}

\newcommand{\Conj}{{\rm Conj.}}

\newcommand{\Def}{{\rm Def.}}

\newcommand{\Prop}{{\rm Prop.}}

\newcommand{\loccit}{{\em loc.\ cit. }}
\newcommand{\loccitn}{{\em loc.\ cit.}}
\newcommand{\cf}{{\em cf. }}
\newcommand{\ie}{i.e.,\ }
\newcommand{\eg}{e.g.,\ }

\newcommand{\diag}{{\rm diag}}

\newcommand{\tor}{{\rm tor}}
\newcommand{\dR}{{\rm dR}}

\DeclareMathOperator{\Div}{div}

\DeclareMathOperator{\Stab}{Stab}

\DeclareMathOperator{\Res}{Res}

\author{Wushi Goldring, Jean-Stefan Koskivirta}
\date{October 1, 2017}
\begin{document}

\begin{abstract}
A general conjecture is stated on the cone of automorphic vector bundles admitting nonzero global sections on schemes endowed with a smooth, surjective morphism to a stack of $G$-zips of connected-Hodge-type; such schemes should include all Hodge-type Shimura varieties with hyperspecial level. We prove our conjecture for groups of type $A_1^n$, $C_2$,  and $\fp$-split groups of type $A_2$ (this includes all Hilbert-Blumenthal varieties and should also apply to Siegel modular threefolds and Picard modular surfaces). An example is given to show that our conjecture can fail for zip data not of connected-Hodge-type.
\end{abstract}

\maketitle

\section*{Introduction}

This paper develops a particular aspect of our general program launched in \cite{Goldring-Koskivirta-Strata-Hasse}. Recall that the aim of the program is to connect (A) {\em Automorphic Algebraicity}, (B) {\em $G$-Zip Geometricity} and (C) {\em Griffiths-Schmid Algebraicity}. This paper zooms in on (B); the basic question which guides us is: 
\begin{qAlph} \label{q-arbitrary-to-gzip} Let $\Zcal$ be a zip datum with underlying group $G$. Assume $X$ is a scheme, endowed with a smooth surjective morphism $\zeta:X \to\GZip^{\Zcal}$. To what extent is the global geometry of $X$ controlled by the stack $\GZip^{\Zcal}$ and the map $\zeta$ ?
\end{qAlph}

We briefly describe a general framework underlying Question~\ref{q-arbitrary-to-gzip}, before specializing to a concrete instance of it which is studied here.
Let $G$ be a connected, reductive $\fp$-group and $\mu \in X_*(G)$ a cocharacter. By the work of Moonen-Wedhorn for $G=GL(n)$ \cite{Moonen-Wedhorn-Discrete-Invariants} and Pink-Wedhorn-Ziegler for general $G$ \cite{PinkWedhornZiegler-F-Zips-additional-structure,Pink-Wedhorn-Ziegler-zip-data}, the pair $(G,\mu)$ gives rise to a zip datum $\Zcal$ and a stack $\GZip^\Zcal$ (see \S\ref{sec-review}). To give some sense of what this stack is, and how it is related to more classical objects, let us recall two historical sources of motivation for it.

One source of motivation comes from Hodge theory; the other from the theory of Shimura varieties and their Ekedahl-Oort (EO) stratification in characteristic $p>0$. Of course these two sources are not disjoint from one another, but they do help to shed light on two different points of view concerning $G$-Zips and their applications. As explained below, the connection with Hodge theory shows that the theory of $G$-Zips can be applied to a very {\em general} class of schemes in characteristic $p$, in the same way that classical Hodge theory applies (at least) to smooth complex projective varieties. By contrast, the connection with the EO stratification of Shimura varieties, also recalled below, gives particularly rich and {\em special} examples of $G$-Zips and is fruitful for applications to automorphic forms (\cf our previous papers \cite{Goldring-Koskivirta-Strata-Hasse,Goldring-Koskivirta-zip-flags}).

\subsection*{Motivation I: Hodge theory} 
\label{sec-motiv-hodge}

As already observed in the introduction to the Moonen-Wedhorn paper \cite{Moonen-Wedhorn-Discrete-Invariants}, the stack $\GZip^\Zcal$ is a characteristic $p$ analogue of a period domain, or more generally a Mumford-Tate domain \cite{Green-Griffiths-Kerr-Mumford-Tate-Domains-book}. 
Suppose $f:Y\to X$ is a proper, smooth morphism of schemes in characteristic $p$ satisfying the conditions of Deligne-Illusie: $\dim(Y/X)<p$ and $f$ admits a lift $\tilde f: \tilde Y \to \tilde X$ with $\tilde X$ flat over $\ZZ/p^2\ZZ$. Then the Hodge-de Rham spectral sequence for $Y/X$ degenerates at $E_1$ and the conjugate spectral sequence degenerates at $E_2$, giving rise to the Hodge and conjugate filtrations respectively. Given $i\geq 0$, the parabolic subgroups $P$ and $Q$ stabilizing the Hodge and conjugate filtrations of $H^i_{\dR}(Y/X)$ give rise to a zip datum $\Zcal$. Thus $H^i_{\dR}(Y/X)$ is a $GL(n)$-Zip of type $\Zcal$, where $n$ is the rank of $H^i_{\dR}(Y/X)$. It yields a map $\zeta:X \to \GLnZip^\Zcal$  \cite[\S6]{Moonen-Wedhorn-Discrete-Invariants}.   

The map $\zeta$ should be thought of as an analogue of the period map associated to a variation of Hodge structure (VHS) over a smooth, projective $\CC$-scheme. This analogy will be the subject of forthcoming work with Y. Brunebarbe. The analogue of our guiding Question~\ref{q-arbitrary-to-gzip} in Hodge theory is one that has played a central role in algebraic geometry for the past 150 or so years, going back (at least) to the work of Abel and Riemann on periods of abelian integrals: To what extent does a period map control the global geometry of the base of a VHS?

If the Hodge and conjugate filtrations preserve certain tensors in $H^i_{dR}(Y/X)^{\langle\otimes\rangle}$, then the above $GL(n)$-Zip arises from a $G$-Zip, where $G$ is the group stabilizing the tensors. For example, Moonen-Wedhorn explain how, when $Y/X$ is a family of $K3$ surfaces, the primitive part of $H^2_{dR}(Y/X)$ is naturally an $SO(2,19)$-Zip. They use this to provide a unified framework for previous works on stratifications of families of K3 surfaces by Artin, Katsura-van der Geer and Ogus (see \cite{Moonen-Wedhorn-Discrete-Invariants,Ogus-height-strata-K3} and the references therein).  In classical Hodge theory, the above brings to mind the Mumford-Tate group. We hope to return to the question of `Mumford-Tate group for $G$-Zips' in future work.

\subsection*{Motivation II: Shimura varieties} 
\label{sec-motiv-shimura}

Let us turn now to motivation stemming from the theory of Shimura varieties and the EO stratification. Let $X$ be the special fiber of the Kisin-Vasiu model of a Hodge-type Shimura variety at a place of good reduction, attached to a Shimura datum $(\GG,\XX)$ and a hyperspecial level at $p$. Write $G:=G_{\ZZ_p}\times \FF_p$, where $G_{\ZZ_p}$ is a reductive $\ZZ_p$-model of $\GG_{\QQ_p}$. If $\Acal/X$ is a universal abelian scheme corresponding to some symplectic embedding of $\gx$, then $H^1_{\dR}(\Acal/X)$ is naturally a $G$-Zip; the classifying map $\zeta:X\to \GZip^\Zcal$ is smooth by  Zhang \cite{ZhangEOHodge} and surjective by Nie \cite{Nie-Newton-EO} and Kisin-Madapusi-Shin \cite{Kisin-Honda-Tate-theory-Shimura-varieties}. 

The EO stratification of $X$ is given by the fibers of $\zeta$. When $X$ is of PEL-type (resp. Siegel type) this recovers the earlier definition of the EO stratification by Moonen \cite{moonen-gp-schemes} (resp. Ekedahl-Oort \cite{Oort-stratification-moduli-space-abelian-varieties}). Even in these special cases the scheme-theoretic structure of the strata and stratification property is most easily seen via the $G$-Zip approach. 

Specializing Question~\ref{q-arbitrary-to-gzip} to the Shimura variety $X$ gives:

\begin{qAlph} To what extent is the global geometry of the Shimura variety $X$ controlled by the stack $\GZip^{\Zcal}$ and the morphism $\zeta$ \label{q-geom-shimura-from-gzip}?
\end{qAlph}

Since the underlying set of the stack $\GZip^\Zcal$ is just a finite set of points, it may initially appear to the reader that a pair $(\GZip^\Zcal, \zeta)$ will capture little of the geometry of $X$. This would suggest that the answer to both Questions~\ref{q-arbitrary-to-gzip} and~\ref{q-geom-shimura-from-gzip} should be "minimal" and that $(\GZip^\Zcal, \zeta)$ retains much less geometric information than a period map in classical Hodge theory.  

One of the key aims of this paper is to provide evidence to the contrary.
The previous papers \cite{Koskivirta-Wedhorn-Hasse,Koskivirta-compact-hodge,Goldring-Koskivirta-Strata-Hasse,Goldring-Koskivirta-zip-flags} already deduced nontrivial information about the global geometry of $X$ from a study of group-theoretical Hasse invariants on $\GZip^{\Zcal}$ (and the closely related stacks of zip flags). For example, we showed by this method that the EO stratification of $X$ is uniformly principally pure \cite[\Cor~3.1.3]{Goldring-Koskivirta-Strata-Hasse} and that all EO strata of the minimal compactification $X^{\min}$ are affine (\loccitn, \Th~3.3.1).  
\subsection*{The global sections cone}
\label{sec-intro-global-sect}
We return to the general setting of Question~\ref{q-arbitrary-to-gzip}: Set $\Xcal:=\GZip^\Zcal$ and consider a characteristic $p$ scheme $X$ equipped with a morphism $\zeta:X \to \Xcal$. 

 This note is concerned with an example where some global geometry of $X$ may be understood purely in terms of $\Xcal$: the question of which automorphic vector bundles $\Vscr_X(\lambda)$ admit global sections on $X$. Our forthcoming joint work with Stroh and Brunebarbe \cite{Brunebarbe-Goldring-Koskivirta-Stroh-ampleness} will study the closely related question of which $\Vscr_X(\lambda)$ are ample on $X$ and on its partial flag spaces (see \S\ref{sec-flag-space} below and \cite{Goldring-Koskivirta-zip-flags}).

Let $L$ be the Levi subgroup of $G$ given by $\Zcal$  and choose a Borel pair $(B,T)$ appropriately adapted to $\Zcal$ (see \S\ref{review}). A $B \cap L$-dominant character $\lambda \in X^*(T)$ gives rise to a vector bundle $\Vscr_{\Xcal}(\lambda)$ on $\Xcal$ (\S\ref{sec-auto-vector-bundles}). 
 Put $\Vscr_X(\lambda):=\zeta^*(\Vscr_\Xcal(\lambda))$. 
 
 We call the $\Vscr_X(\lambda)$ the \emph{automorphic vector bundles} associated to $\zeta$.
 When $X$ is the special fiber of a Hodge-type Shimura variety, the $\Vscr_X(\lambda)$ recover the usual automorphic vector bundles on $X$.  For a general $X$, it is a priori unclear what, if any, relationship the $\Vscr_X(\lambda)$ bear to automorphic forms. 

Let $C_\Xcal$ (resp. $C_X$) denote the (saturated) cones of $\lambda\in X^*(T)$ such that $\Vscr(n\lambda)$ (resp. $\Vscr_X(n\lambda)$) admits a nonzero global section for some $n\geq 1$ (\S\ref{subsec-cones}). 
The inclusion $ C_\Xcal  \subset C_X $ holds in general simply by pulling back sections. 
Below we propose a conjecture that, under certain hypotheses, the global sections cones of $\Xcal$ and $X$ are equal. This is surprising because one expects bundles on $X$ to admit many more sections than bundles on $\Xcal$; nevertheless our conjecture predicts that the mere existence of sections is to a large extent controlled by $\Xcal$. 

Our approach to the conjecture, as well as one of the hypotheses, will be in terms of the stack of zip flags $\Ycal \to \Xcal$ and the flag space $Y=X \times_\Xcal \Ycal$, both recalled in \S\ref{sec-flag-space}. The stack $\Ycal$ admits a stratification parameterized by the Weyl group of $T$ in $G$; if $\zeta$ is smooth then the same is true of $Y$ by pullback. 

We will say that a reduced scheme $S$ is pseudo-complete if every $h \in H^0(S, \Ocal_S)$ is locally constant. For example, a proper, reduced scheme is pseudo-complete.

\begin{conjAlph}[Conjecture~\ref{conj}]
\label{conj-intro} Let $\zeta:X \to \GZip^\Zcal$.
Assume that:  
\begin{enumerate}[(a)]
\item \label{item-intro-hodgetype} The zip datum $\Zcal$ is of connected-Hodge-type \textnormal{(\Def~\ref{def-conn-hodge-type})}.
\item \label{item-intro-assume-smooth} For all connected components $X^\circ \subset X$, the map $\zeta:X^\circ \to \GZip^\Zcal$ is smooth and surjective.
\item 
\label{item-intro-assume-loc-const} 
All strata closures in $Y$ are pseudo-complete.
\end{enumerate}
Then the global sections cones of $X$ and $\Xcal$ coincide: $ C_\Xcal  =  C_{X} $.\end{conjAlph}

The following result establishes the conjecture in some special cases. 

\begin{thmAlph}[Theorems \ref{mainthm},~\ref{th-further}]
\label{th-intro}
Suppose that either \begin{enumerate}[(a)]
\item
\label{item-type-a1}
$G$ is of type $A_1^n$ (\ie $G^{\ad}_{\fpbar}\cong PGL(2)_{\fpbar}^n$) and $\Zcal$ is attached to a Borel of $G$,
\item
\label{item-type-c2}
$G$ is of type $C_2$ and $\Zcal$ is proper of connected-Hodge-type \textnormal{(\Def~\ref{def-conn-hodge-type})},
\item 
\label{item-type-a2}
$G$ is $\fp$-split of type $A_2$ and $\Zcal$ is proper of connected-Hodge-type.
\end{enumerate}
Then \textnormal{Conjecture~\ref{conj-intro}} holds for $\Zcal$.
\end{thmAlph}
 In the three cases of Theorem~\ref{th-intro}, $ C_\Xcal$ is given explicitly in Corollary \ref{cor Hilbert} and Figures~\ref{fig-A2},~\ref{fig-C2} respectively. 

For example, Conjecture~\ref{conj-intro} applies to a proper smooth $k$-scheme $X$ endowed with a smooth, surjective map $X\to \GZip^\Zcal$. It should also apply when $X$ is the special fiber at $p$ of a Shimura variety of Hodge-type with hyperspecial level at $p$ (see \S\ref{sec-shimura}). Specializing to this case, Theorem~\ref{th-intro}\eqref{item-type-a1} applies to Hilbert modular varieties. Modulo a technical assumption on toroidal compactifications, part~\eqref{item-type-c2} applies to Siegel modular threefolds (Shimura varieties of type $GSp(4)$) and part~\eqref{item-type-a2} applies to Picard modular surfaces at a split prime ($GU(2,1)$-Shimura varieties at a split prime). However, as emphasized above, the range of applications of both the conjecture and the theorem is much broader than just Shimura varieties. 

Both $\Xcal, \Ycal$ are stratified (\S\ref{sec-flag-space}) and when $\zeta$ is smooth, so too are $X,Y$. They are then the Zariski closures of their top-dimensional strata. A natural generalization of Conjecture~\ref{conj-intro} is to ask when the global sections cone $C_{\Xcal,w}$ of a stratum $\Xcal_w$ of $\Xcal$ coincides with the cone $C_{X,w}$ of the corresponding stratum $X_w$ of $X$. We find it more natural to study the analogous question on the flag space $Y \to X$, see Question~\ref{q-strata-cones}. The situation for general strata seems more complicated than for $X$ itself, see Remark~\ref{rmk-variants-conj}\eqref{item-fail-strata}. Nevertheless, when $G$ is of type $A_1^n$, we define a notion of "admissible stratum" (\Def~\ref{defadm}) and prove the following:
\begin{thmAlph}[Theorem~\ref{mainthm}] 
\label{th-intro-strata}
Let $G$ be a group of type $A_1^n$ and $\zeta:X \to \GZip^\Zcal$ as in \Th~\ref{th-intro}\eqref{item-type-a1}. Then $ C_{\Xcal,w} = C_{X,w} $ holds for each admissible stratum $\Xcal_w$. Moreover, if $G$ is $\fp$-split, then all strata are admissible.
\end{thmAlph}
Note that $X=Y$ in the context of \Th~\ref{mainthm}. See \Th~\ref{th-further} for related results about the equality of cones of flag strata in case $G$ is of type $C_2$ or split of type $A_2$.

F. Diamond shared with us his conjecture that when $X$ is a Hilbert modular variety, the cone $ C_X $ is equal to that spanned by Goren's partial Hasse invariants \cite{Goren-partial-hasse}. As Diamond later informed us, a related question of determining the 'minimal cone' of mod $p$ Hilbert modular forms had been raised earlier by Andreatta-Goren \cite[Question 15.8]{Andreatta-Goren-low-dimension}. Inspired by Diamond's Conjecture and the observation that Goren's cone can be reinterpreted as the Zip cone $ C_\Xcal $, we were led to study Conjecture~\ref{conj-intro}, first for groups of type $A_1^n$ and then more generally.
 
After we announced the results of this paper and communicated them to Diamond, we received the preprint of Diamond-Kassaei \cite{Diamond-Kassaei}. 
It determines the `minimal cone' of Hilbert modular forms mod $p$. 
As a corollary Diamond-Kassaei deduce a different proof that $ C_\Xcal = C_X $ in the special case when $X$ is the special fiber of a Hilbert modular variety at a place of good reduction (\Cor~1.3 of \loccitn). 

The approach of \loccit uses special properties of Hilbert modular varieties (\eg the results of Tian-Xiao that in the Hilbert case EO strata are $\po$-bundles over quaternionic Shimura varieties). By contrast, our methods only use the map $\zeta$ and its basic properties, which hold for all Hodge-type Shimura varieties and even more general $\GZip^{\Zcal}$-schemes. It remains to be seen whether the Diamond-Kassaei result on the `minimal cone' holds in our more general setting, or whether this finer information is special to Hilbert modular varieties.

\subsection*{Outline} \S\ref{sec-review} recalls the theory of $G$-zips, $G$-zip-flags and automorphic vector bundles in this context. In \S\ref{sec general setting}, we define the global sections cones in $X^*(T)$ and formulate our conjectural generalization of Theorem~\ref{th-intro} to zip data of Hodge type, see \Conj~\ref{conj}. \S\ref{sec-strategy} gives some general results which form the basic strategy for proving Theorem~\ref{th-intro}. The proof of Theorem~\ref{th-intro}\eqref{item-type-a1} is the subject of \S\ref{sec-Hilbert}. Our results on groups of type $A_2,C_2$ and $C_3$ are given in \S\ref{sec-further}.
\section*{Acknowledgments}
We thank Fred Diamond for sharing his conjecture with us and both Diamond and Payman Kassaei for helpful discussions and correspondence. We are grateful to Yohan Brunebarbe and Beno\^{i}t Stroh for our collaboration on related questions concerning ampleness of automorphic bundles and period maps in the setting of $G$-Zips. We thank Torsten Wedhorn for helpful comments on an earlier version of this paper.

W. G. thanks the University of Zurich for providing excellent working conditions and the opportunity to present some of the results of this paper during a visit in the Fall of 2016.

Finally, we thank the referees for their helpful comments.
\section{Review of Zip data, flag spaces and automorphic bundles} \label{sec-review}

\subsection{Zip data (\cite{PinkWedhornZiegler-F-Zips-additional-structure,Pink-Wedhorn-Ziegler-zip-data})} 
\label{review}
Fix an algebraic closure $k$ of $\FF_p$. Let $G$ be a connected reductive $\fp$-group. Denote by $\varphi:G\to G$ the Frobenius morphism. Let $\Zcal:=(G,P,L,Q,M,\varphi)$ be a Frobenius zip datum. Recall that this means $P,Q$ are parabolic subgroups of $G_k$ and $L\subset P$, $M\subset Q$ are Levi subgroups, with the property that $\varphi(L)=M$. We say that $\Zcal$ is a zip datum of Borel-type if $P$ is a Borel subgroup of $G$ (this implies that $Q$ is a Borel too). The zip group $E$ is the subgroup of $P\times Q$ defined by
\begin{equation}
E:=\{(x,y)\in P\times Q, \ \varphi(\overline{x})=\overline{y}\}
\end{equation}
where $\overline{x}\in L$ and $\overline{y}\in M$ denote the Levi components of $x,y$ respectively. Let $G\times G$ act on $G$ by $(a,b)\cdot g:=agb^{-1}$; restriction yields an action of $E$ on $G$. The stack of $G$-zips of type $\Zcal$ is isomorphic to the quotient stack $\GZip^\Zcal \simeq \left[E\backslash G\right]$. We say that $\Zcal$ is {\em proper} if $P$ is a proper parabolic subgroup of $G$.

For convenience, we assume that there exists a Borel pair $(B,T)$ defined over $\FF_p$ such that $B\subset P$. Then there exists an element $z\in W$ such that ${}^zB\subset Q$, and $(B,T,z)$ defines a $W$-frame for $\Zcal$ (\cite[\Def~2.3.1]{Goldring-Koskivirta-zip-flags}).

A cocharacter datum $(G, \mu)$ is a connected, reductive $\fp$-group $G$ together with $\mu \in X_*(G)$. Every such $(G, \mu)$ gives rise to a zip datum $\Zcal_\mu$ (\cite[\S2.2]{Goldring-Koskivirta-zip-flags}). Given a cocharacter datum $(G,\mu)$, one has the associated adjoint datum $(G^\ad, \mu^\ad)$, where $G^\ad$ is the adjoint group of $G$ and $\mu^\ad$ is the composition of $\mu$ with $G \twoheadrightarrow G^\ad$.   

\begin{definition}
\label{def-conn-hodge-type}
Let $(G, \mu)$ be a cocharacter datum. Let $PSp(2g)$ be the split adjoint group of type $C_g$ and $\mu_g \in X_*(PSp(2g))$ a minuscule cocharacter. We say that $(G, \mu)$ is of \underline{connected-Hodge-type} if for some $g \geq 1$, the adjoint datum admits an embedding $(G^\ad, \mu^\ad) \hookrightarrow (PSp(2g), \mu_g)$. A zip datum $\Zcal$ is of connected-Hodge-type if $\Zcal=\Zcal_\mu$ for some $(G,\mu)$ of connected-Hodge-type.
\end{definition}

\subsection{Notation}
\label{sec-notation}
Let $\Phi\subset X^*(T)$ (resp. $\Phi_L$) be the set of $T$-roots in $G$ (resp. $L$). Let $\Phi^+$ (resp. $\Phi^+_L$) be the system of positive roots given by putting $\alpha \in \Phi^+$ (resp. $\alpha \in \Phi^+_L$) when the $(-\alpha)$-root group $U_{-\alpha}$ is contained in $B$ (resp. $B_L:=B\cap L$). Write $\Delta\subset \Phi^+$ (resp. $I \subset \Phi_L^+$) for the subset of simple roots. 

For $\alpha \in \Phi$, let $s_\alpha$ be the corresponding root reflection. Let $W$ (resp. $W_L$) be the Weyl group of $\Phi$ (resp. $\Phi_L$). 
Then $(W,\{s_\alpha| \alpha \in \Delta\})$ is a Coxeter system; denote by $\ell :W\to \NN$ its length function and by $\leq$ the Bruhat-Chevalley order. 
Write $w_0$ for the longest element of $W$. The \emph{lower neighbors} of $w\in W$ are the $w' \in W$ satisfying $w' \leq w$ and $\ell(w')=\ell(w)-1$.
Let ${}^{I} W \subset W$ be the subset of elements $w\in W$ which are minimal in the coset $W_L w$. The $I$-dominant characters of $T$ are denoted $X^*_{+,I}(T)$. 

The Zariski closure of a subscheme or substack $Z$ is denoted $\overline{Z}$; it is always endowed with the reduced structure.

\begin{comment}

\subsubsection{Zip stratification}
Let $I\subset \Delta$ (resp. $J\subset \Delta$) be the type of $P$ (resp. $Q$). For a subset $K\subset \Delta$, let $W_K \subset W$ be the subgroup generated by $\{s_\alpha,\alpha \in K$.  Let $w_0$ (resp. $w_{0,K}$) be the longest element in $W$ (resp. $W_K$). Denote by ${}^K W$ (resp. $W^K$) the subset of elements $w\in W$ which are minimal in the coset $W_K w$ (resp. $wW_K$). The set ${}^K W$ (resp. $W^K$) is a set of representatives for the quotient $W_K \backslash W$ (resp. $W/W_K$). For $w\in W$, fix a representative $\dot{w}\in N_G(T)$, such that $(w_1w_2)^\cdot = \dot{w}_1\dot{w}_2$ whenever $\ell(w_1 w_2)=\ell(w_1)+\ell(w_2)$ (this is possible by choosing a Chevalley system, see \cite[Exp. XXIII, \S6]{SGA3}).

For $h\in G(k)$ let $\Ocal_\Zcal(h)$ be the $E$-orbit of $h$ in $G$. For $w\in W$, define $G^w:=\Ocal_\Zcal(\dot{w}\dot{z}^{-1})$. By \cite[\Ths~7.5,11.3]{Pink-Wedhorn-Ziegler-zip-data}, the map $w\mapsto G^w$ restricts to two bijections:
\begin{align}\label{param}
{}^I W \to \{E \textrm{-orbits in }G\} \\
W^J \to \{E \textrm{-orbits in }G\}.
\end{align}
Furthermore, one has the following dimension formula
\begin{equation}\label{dimzipstrata}
\dim(G^w)= \ell(w)+\dim(P)\quad \textrm{for all }w\in {}^I W \cup W^J
\end{equation}

\end{comment}

\subsection{Review of flag spaces and their stratification (\cite{Goldring-Koskivirta-Strata-Hasse,Goldring-Koskivirta-zip-flags})} \label{sec-flag-space}
Let $(B,T)$ be an $\fp$-Borel pair of $G$ such that $B\subset P$ (this can be assumed after possibly conjugating $\Zcal$, see \cite[\Rmk~1.3.2(2)]{Goldring-Koskivirta-zip-flags}). Write $\Xcal:=\GZip^\Zcal$. The stack of zip flags $\Ycal:=\GF^{\Zcal}$ was defined in \cite[\S5.1]{Goldring-Koskivirta-Strata-Hasse}; see also \cite[\S3]{Goldring-Koskivirta-zip-flags}. It is isomorphic to $[E'\backslash G]$ where $E'=E\cap (B\times G)$. Recall that $\Ycal$ parametrizes $G$-zips with an additional compatible $B$-torsor. Thus $\Ycal$ is naturally a $P/B$-bundle $\pi:\Ycal \to \Xcal$.

Consider a morphism of stacks $\zeta:X\to \Xcal$. Form the fiber product
\begin{equation} \label{eq-def-flag-space}
\xymatrix@1@M=3pt{
Y \ar[r]^-{\zeta_Y} \ar[d]_-{\pi_{Y/X}} & \Ycal \ar[d]^-{\pi} \\
X \ar[r]_-{\zeta} & \Xcal
}
\end{equation} 
We call $Y$ the (full) flag space of $X$ attached to $B$ (\cite[\S10.3]{Goldring-Koskivirta-Strata-Hasse}). In \cite[\S7.2]{Goldring-Koskivirta-zip-flags}, we also  defined partial flag spaces for intermediate parabolics $B\subset P_0\subset G$, but these are not used here. By \cite[\S4.1]{Goldring-Koskivirta-zip-flags}, there is a zip datum $\Zcal_B:=(G,B,T,{}^z B,T,\varphi)$ and natural smooth morphisms of stacks $\Psi$ and $\beta$ as follows:
\begin{equation}
\label{stratif-mor}
\Ycal \xrightarrow{\Psi} \GZip^{\Zcal_B} \xrightarrow{\beta}[B\backslash G /B].
\end{equation}
The stacks $\Xcal_B:=\GZip^{\Zcal_B}$ and $\Sbt:=[B\backslash G /B]$ are finite; their points are both parametrized by the Weyl group $W$. The stack $\Sbt$ admits the \emph{Schubert stratifiation} by locally closed substacks $\Sbt_w$ for $w \in W$ ordered by the Bruhat-Chevalley order.  The morphism $\beta$ is bijective, but not an isomorphism. By pullback, the fibers of $\Psi$ define a stratification of $\Ycal$ by locally closed substacks $\Ycal_w$, with the same closure relations.

Let $Y_{w}:=\zeta_Y^{-1}(\Ycal_w)$, the corresponding flag stratum in $Y$. Both $\Ycal_w$ and $Y_w$ are endowed with the reduced structure. The Zariski closure $\overline{\Ycal}_{w}$ of $\Ycal_{w}$ is normal (\cite[\S4]{Goldring-Koskivirta-zip-flags}). Let $Y^*_{w}:=\zeta_Y^{-1}(\overline{\Ycal}_{w})$. If $\zeta$ is smooth, then so is $\zeta_Y$ and then $Y_w^*=\overline{Y}_w$.  If $\zeta$ is not smooth,  $Y_w^*$ may not be the Zariski closure of $Y_w$.

Although we shall not need it explicitly in this paper, recall that $\GZip^\Zcal$ also admits a `zip stratification',  whose strata are parameterized by ${}^IW$ (\cite{PinkWedhornZiegler-F-Zips-additional-structure,Pink-Wedhorn-Ziegler-zip-data}). When $G$ is of type $A_1^n$ and $\Zcal$ is of Borel type, $\GZip^\Zcal=\GF^\Zcal$ and the zip stratification agrees with the one of $\GF^\Zcal$ recalled above. 
\subsection{Automorphic vector bundles}
\label{sec-auto-vector-bundles}
All of the automorphic bundles studied in this paper arise from the general associated sheaves construction: If a $k$-group $H$ acts on a $k$-scheme $X$, then every $H$-representation $\rho$ on a $k$-vector space yields a vector bundle $\Vscr(\rho)$ on the quotient stack $[H \backslash X]$, \cf \cite[\S{N.3}]{Goldring-Koskivirta-Strata-Hasse} and \cite[\S5.8]{jantzen-representations}. In particular, every representation of $E$ (resp. $E'$, $B$, $B \times B$)  yields an associated vector bundle on   $\Xcal=[E\backslash G]$ (resp. $\Ycal=[E' \backslash G]$, $P/B$, $[B  \backslash G/B ]$).   

A character $\lambda\in X^*(T)$ gives a $P$-equivariant line bundle $\Lscr_{\lambda}$ on the flag variety $P/B$. The $P$-module $H^0(\lambda):=H^0(P/B, \Lscr_{\lambda})$ gives an $E$-module via the first projection $E \to P$. Denote by $\Vscr_\Xcal(\lambda)$ the associated vector bundle on $\Xcal$.  If $\lambda\in X^*(T)$ is not $I$-dominant, then $\Vscr_\Xcal(\lambda)=0$ by definition.
Given a stack $X$ and a morphism $\zeta:X \to \Xcal$, set $\Vscr_X(\lambda):=\zeta^*(\Vscr(\lambda))$. We call the $\Vscr_\Xcal(\lambda)$  \emph{automorphic vector bundles}.

Let $\Lscr_{\Ycal}(\lambda)$ be the line bundle on $\Ycal$ associated to $\lambda$ via the first projection $E' \to B$. Set $\Lscr_Y(\lambda):=\zeta_Y^{*}(\Lscr_{\Ycal}(\lambda))$. One has the direct image formulas:
\begin{equation}
\label{eq-llambda-vlambda}
(\pi_{Y/X})_*\Lscr_{Y}(\lambda)=\Vscr_X(\lambda).
\end{equation}

\begin{comment}

\subsubsection{Characteristic sections on strata}
Assume $\lambda \in X^*(L)$, so $\Vscr(\lambda)$ is a line bundle on $\GZip^\Zcal$. Recall that, for any $w\in {}^I W$, there exists $n\geq 1$ such that for all $\lambda \in X^*(L)$, the line bundle $\Vscr(n\lambda)=\Vscr(\lambda)^{\otimes n}$ admits a nonzero section (unique up to a nonzero scalar)
\begin{equation}\label{natseczip}
f_{w,\lambda}\in H^0(\left[E \backslash G_w\right],\Vscr(n\lambda)).
\end{equation}

In \cite[\Th~3.1.2]{Goldring-Koskivirta-Strata-Hasse} and \cite[\Th~5.1.6]{Goldring-Koskivirta-zip-flags} we proved the following: If $\lambda$ is $\Zcal$-ample (\loccit \Def~5.1.3) and orbitally $p$-close (\loccit \Def~5.1.2), there exists $N\geq 1$ such that for all $w\in {}^I W$ the section $f^N_{w,\lambda}$ extends over the Zariski closure $\left[ E\backslash \overline{G}_w\right]$ with non-vanishing locus exactly $\left[E\backslash G_w\right]$. We refer to the $f^N_{w,\lambda}$ as \emph{group-theoretical Hasse invariants} or \emph{characteristic sections}.

Furthermore, if $\Zcal$ is a zip datum attached to a maximal cocharacter-datum $(G,\mu)$ (\loccit \S2.4) and $\lambda$ is maximal (\cite[\Def~2.4.3]{Goldring-Koskivirta-zip-flags}), then a power of $f_{w,\lambda}$ extends as above (\loccit \Cor~6.2.2). In particular, this applies to the Hodge line bundle in the context of Shimura varieties of Hodge-type (\loccit \Rmk~2.4.4).

\end{comment}

\section{The conjecture}\label{sec general setting}
\subsection{Cones} \label{subsec-cones}
Let $G$ be a connected, reductive $\fp$-group. Fix a zip datum $\Zcal:=(G,P,L,Q,M,\varphi)$, with an $\FF_p$-Borel pair $(B,T)$ such that $B\subset P$. Recall that $\Xcal:=\GZip^{\Zcal}$ and $\Ycal=\GF^\Zcal$ denote the associated stacks of $G$-zips and $G$-zip flags. Let $X$ be a stack together with a map $\zeta : X\to \Xcal$. The trivial example $X=\Xcal$ is allowed here. Let $\zeta_Y:Y \to \Ycal$ be the base change of $\zeta$ by $\pi: \Ycal \to \Xcal$. 
\begin{definition} \label{def-global-sections-cone} 
For $w \in W$, the global sections cones of $X$ and $Y_w$ are
\begin{equation}
\label{eq-global-sections-cone-X}
 C_{X} :=\{\lambda \in X^*(T) \ | \  H^0(X, \Vscr_X(n\lambda)) \neq 0 \mbox{ for some }n \geq 1\} 
\end{equation}
\begin{equation}
\label{eq-global-sections-cone-strata}
 C_{Y,w} :=\{\lambda \in X^*(T) \ | \ H^0(Y^*_w, \Lscr_Y(n\lambda)) \neq 0 \mbox{ for some }n \geq 1\} 
\end{equation}
\end{definition}

Put $C_Y=C_{Y,w_0}$. By~\eqref{eq-llambda-vlambda}, one has  $H^0(Y,\Lscr_Y(\lambda))=H^0(X,\Vscr_X(\lambda))$; thus $C_X=C_Y$. If $\zeta$ is surjective, so is $\zeta_{Y}$ and then $C_{\Ycal,w} \subset  C_{Y,w} $ for all $w\in W$.

The main focus of this paper is the following instance of Question~\ref{q-arbitrary-to-gzip}:
\begin{question}
\label{q-strata-cones} 
For which $w\in W$ is $C_{\Ycal,w} = C_{Y,w} $?
\end{question}
\begin{definition}
\label{def-pseudo-complete}
A reduced scheme $Z$ is \underline{pseudo-complete} if every $h\in H^0(Z,\Ocal_Z)$ is locally constant.
\end{definition}  Concerning the cone $ C_X =C_Y=C_{Y,w_0}$, our principal conjecture is:

\begin{conjecture} \label{conj}
Let $X$ be a $k$-scheme and $\zeta: X \to \Xcal$. Assume that:  
\begin{enumerate}[(a)]
\item \label{hodgetype} The zip datum $\Zcal$ is of connected-Hodge-type \textnormal{(\Def~\ref{def-conn-hodge-type})}.
\item \label{item-assume-smooth} For any connected component $X^\circ \subset X$, the map $\zeta:X^\circ \to \Xcal$ is smooth and surjective.
\item \label{item-assume-loc-const} For all $w\in W$, $Y^*_{w}$ is pseudo-complete.
\end{enumerate}
Then the global sections cones of $X$ and $\Xcal$ coincide, that is $C_X= C_{\Xcal} $.\end{conjecture}

\begin{rmk}
\label{rmk-variants-conj}
As motivation for the assumptions of the conjecture, the following notes how some variants fail to hold.
\begin{enumerate}[(a)] 
\item A multiple $n \geq 1$ as in \Def~\ref{def-global-sections-cone} is necessary. For example, the special fiber of the (compactified) modular curve satisfies the assumptions of~\ref{conj}. In this case, the Hodge line bundle $\omega$ satisfies
$H^0(\Xcal,\omega^n)\neq 0 \Longleftrightarrow n=(p-1)m, m\geq 0$. \item \label{item-fail-strata}
We conjecture that $ C_{\Ycal,w}  \neq  C_{Y,w}$ when $Y$ is the mod $p$ special fiber of a Hilbert modular threefold at a totally inert prime and $w \in W$ has length two. This conjecture would show that the answer to Question~\ref{q-strata-cones} can be "not all".
\item 
\label{item-fail-non-hodge}
In \S\ref{sec-gsp6}, we give an example of a pair $(X,\zeta)$ satisfying assumptions \ref{conj}\eqref{item-assume-smooth}-\eqref{item-assume-loc-const}, but not \eqref{hodgetype}, for which $C_{\Xcal}  \neq  C_{X}$. 
\end{enumerate}
\end{rmk}
\begin{rmk} 
In contrast with $C_X  =C_Y $, when $\Ycal \neq \Xcal$ it seems difficult to relate the cones of flag strata in $Y$ (resp. $\Ycal$) with cones of zip strata in $X$ (resp. $\Xcal$). For this reason, we don't know if it's reasonable to expect a variant of Conjecture~\ref{conj}, where~\eqref{item-assume-loc-const} is replaced by the analogous condition for strata of $X$.   
\end{rmk}
\subsection{Shimura varieties}
\label{sec-shimura}
Let $X$ be the special fiber of a Hodge-type Shimura variety with hyperspecial level at $p$. Let $G$ be the corresponding reductive $\fp$-group and $\Zcal$ the zip datum of $X$. By \cite{ZhangEOHodge}, there is a smooth morphism of stacks
$\zeta:X\to \GZip^\Zcal$.
Let $X^{\tor}$ be a smooth, projective toroidal compactification of $X$ afforded \cite{MadapusiHodgeTor}. In \cite[\S5.1]{Goldring-Koskivirta-Strata-Hasse}, we constructed an extension $\zeta^{\tor}:X^{\tor} \to \GZip^\Zcal$ of $\zeta$ to $X^{\tor}$. 

By definition, both $\zeta$ and $\zeta^{\tor}$ satisfy \ref{conj}\eqref{hodgetype}. A number of works have recently shown that $\zeta$ is surjective on every connected component $X^{\circ}$ of $X$ \cf~\cite{Lee-newton-strata-nonempty,Kisin-Honda-Tate-theory-Shimura-varieties}. Thus $\zeta$ satisfies~\ref{conj}\eqref{item-assume-smooth}. Since $X^\tor$ is reduced and proper, $(X^\tor, \zeta^{\tor})$ satisfies~\ref{conj}\eqref{item-assume-loc-const}.

Therefore, if $\zeta^{\tor}$ is smooth, then Conjecture~\ref{conj} applies to $(X^\tor, \zeta^\tor)$. Moreover, provided the usual hypotheses are satisfied, the classical Koecher principle implies that $ C_X = C_{X^{\tor }}$, so the conjecture for $X^{\tor}$ is equivalent to that for $X$.

The smoothness of $\zeta^\tor$ should follow from the work of Lan-Stroh \cite{Lan-Stroh-stratifications-compactifications}. For the special groups appearing in Theorem~\ref{th-intro}, the smoothness of $\zeta^{\tor}$ may also follow from Boxer's thesis \cite{Boxer-thesis}, once it is suitably reinterpreted in the language of $G$-Zips. We expect that ~\ref{conj}\eqref{item-assume-loc-const} for $X$ itself also follows from a version of Lan-Stroh's Koecher principle for strata  \cite[\Th~2.5.10]{Lan-Stroh-stratifications-compactifications}, but have not checked this. 

In any case, \ref{conj}\eqref{item-intro-assume-loc-const} certainly holds for Hilbert modular varieties $X$ of dimension $>1$ by the classical Koecher principle, because then $X=Y$ is its own flag space and the proper strata of $X$ are proper (they do not intersect the toroidal boundary).

\section{Strategy of proof} \label{sec-strategy}
\subsection{Some general remarks}
\label{sec-observe}
Assume $X$ is a $k$-scheme satisfying~\ref{conj}\eqref{item-intro-assume-smooth}-\eqref{item-intro-assume-loc-const}.
Proposition~\ref{prop-codim-1} and Corollary~\ref{cor-codim-1} below provide a simple strategy to prove the equality of cones $ C_{\Ycal,w}  = C_{Y,w}$ for all $w \in W$. 
This strategy assumes that the stratification of $\Ycal$ has some particularly nice properties. {\em A priori}, it supposes neither that $\Zcal$ is of connected-Hodge-type, nor does it use that $Y$ arises as a fiber product of $X$ and $\Ycal$ over $\Xcal$.

The problem is then that the hypotheses of Proposition~\ref{prop-codim-1} will usually not be satisfied by all strata. 
In Theorem~\ref{th-intro}, the only cases where the hypotheses below are satisfied for all $w \in W$ are $\fp$-split groups of type $A_1^n$ and arbitrary groups of type $A_1 \times A_1$. The work to prove Theorem~\ref{th-intro} in the other cases consists of weakening the hypotheses of Proposition~\ref{prop-codim-1} and using additional knowledge about the cone $ C_Y$. The former leads to the notions of admissibility in \S\ref{sec-result-hilbert}; the latter uses the fact that, since $Y \to X$ is a flag variety bundle, $C_Y \subset X^*_{+,I}(T)$. Proposition~\ref{colength1} gives a  simple but useful extension of this kind. 
\subsection{One-dimensional strata}
\label{sec-one-dim-strata}
The following proposition will serve as the first step of many inductive arguments later on. 
\begin{proposition} Assume $\zeta:X \to \Xcal$ satisfies~\textnormal{\ref{conj}\eqref{item-intro-assume-smooth}-\eqref{item-intro-assume-loc-const}}.
\label{onlyone}
If $w \in W$ and $\ell(w)=1$, then $C_{\Ycal,w} = C_{Y,w} $.
\end{proposition}
\begin{proof}
Let $\lambda \in  C_{Y,w} $ and assume $\lambda \notin  C_{\Ycal, w}$. Since $\ell(w)=1$, we have $-\lambda \in  C_{\Ycal, w}$. Hence for some $m \geq 1$, there exist nonzero $h\in H^0(\Ycal^*_w,\Lscr_{\Ycal}(-m\lambda))$ and $f\in H^0(Y^*_w, \Lscr_Y(m\lambda))$. 

Since $\zeta_Y$ is smooth, $Y^*_w$ is reduced, so there is an irreducible component $Y_w'\subset Y_w$ where $f|_{Y'_w}\neq 0$. Since $h$ is nowhere vanishing on $\Ycal_w$, the pullback $\zeta_Y^*(h)$ is nowhere vanishing on $Y_w$. In particular, $\zeta_Y^*(h)$ is nowhere zero on $Y'_w$. So $\zeta_Y^*(h)f\in H^0(Y'_w,\Ocal_{Y'_w})$ is nonzero too. By \ref{conj}\eqref{item-assume-loc-const}, $\zeta_Y^*(h)f$ is constant. Thus $h$ is nowhere zero on $\Ycal_w$ and $\Lscr_{\Ycal}(m\lambda)|_{\Ycal_w}\simeq \Ocal_{\Ycal_w}$; this contradicts $\lambda \notin  C_{\Ycal,w}$. 
\end{proof}

\subsection{Changing the center of $G$} \label{sec-change-group}

Let $\tilde{G}$ be the simply-connected covering of the derived group of $G$ (in the sense of reductive algebraic groups). Write $\iota:\tilde G \to G$ for the natural map. Pulling back $\Zcal$ to $\tilde G$ along $\iota$ yields a zip datum $\tilde{\Zcal}$ for $\tilde{G}$. Write $\tilde{\Xcal}=\GtildeZip^{\tilde{\Zcal}}$ and $\tilde{\Ycal}$ for the corresponding stack of zip flags. The map $\iota : \tilde{G}\to G$ induces a homeomorphism $\tilde{\Xcal}\to \Xcal$. Consider the fiber product
$$\xymatrix@1@M=5pt{
\tilde{X}\ar[r]^-{\tilde{\zeta}} \ar[d]^{\iota_X}& \tilde{\Xcal}\ar[d]^\iota\\
X\ar[r]^-{\zeta} & \Xcal.
}$$

Put $\tilde{T}=\iota^{-1}(T)$, a maximal torus of $\tilde{G}$. Write $\iota^*:X^*(T) \rightarrow X^*(\tilde T)$ for the restriction map. 

\begin{lemma}
\label{change}
Let $X$ be a stack and $\zeta:X \to \Xcal$ arbitrary ($X=\Xcal$ allowed).
For all $w\in W$, one has 
$ \iota^* C_{Y,w}  = C_{\tilde{Y}, w}$. In particular,
$ C_{\Ycal, w} = C_{Y,w} \Longleftrightarrow  C_{\tilde{\Ycal}, w}= C_{\tilde{Y},w} .$
\end{lemma}
\begin{proof} By pullback of sections, $\iota^* C_{Y,w}  \subset C_{\tilde{Y},w} $ for all $w \in W$. The reverse inclusions follow from the descent lemma \cite[3.2.2]{Goldring-Koskivirta-Strata-Hasse}.
\end{proof}
Consequently, the equality of global sections cones for $(X, \zeta)$ depends only on the type of $G$, not on $G$ itself.
\subsection{The basic strategy}
\label{sec-basic-strategy}
\begin{definition}
\label{def-partial-hasse-schubert}
Let $w \in W$ and $\lambda \in X^*(T)$.
\begin{enumerate}[(a)]
\item A \underline{partial Hasse invariant} of $\Lscr_{\Ycal}(\lambda)$ on $\Ycal^*_{w}$ is a section $s \in H^0(\Ycal^*_{w}, \Lscr_{\Ycal}(\lambda))$ which is pulled back from the Schubert stratum $\Sbt^*_w$ \textnormal{(\S\ref{sec-flag-space})}.
\item The \underline{Schubert cone} $C_{\Sbt,w} \subset C_{\Ycal,w}$ of $w$ is the cone of $\lambda \in X^*(T)$ such that $\Lscr_\Ycal(N\lambda)$ admits a partial Hasse invariant on $\Ycal_w$ for some $N\geq 1$.
\end{enumerate}
\end{definition}
Recall that $\Ycal^*_{w}$ and $Y^*_w$ are normal, so we may consider Weil divisors. If $s \in H^0(\Ycal^*_{w}, \Lscr_{\Ycal}(\lambda))$ is a partial Hasse invariant, its divisor will be supported on a (possibly empty) union of codimension one strata closures in $\Ycal^*_{w}$. If $\zeta$ is smooth, the multiplicities in $\Div(\zeta_Y^*(s))$ equal those of $\Div(s)$ .

\begin{definition} \label{def-separating}
Let $w\in W$ and $\{w_i\}_{i=1}^n$ the set of lower neighbors of $w \in W$. A \underline{separating system} of partial Hasse invariants for $\Ycal_{w}$ is a set of partial Hasse invariants $\{s_i\}_{i=1}^n$ with $s_i \in H^0(\Ycal^*_{w}, \Lscr_{\Ycal}(\lambda_i))$ such that $\Div(s_i)=\Ycal^*_{w_i}$.
\end{definition}
 
\begin{rmk} \
\label{rmk-sep-basis}
\begin{enumerate}[(a)] 
\item For groups of type $A_1^n$, there always exists a particularly simple separating system of partial Hasse invariants, see \S\ref{thegroup}. 
\item In general, many strata do not admit a separating system, as the number of lower neighbors of $w \in W$ can exceed the semisimple rank of $G$.
\item If $\Pic(G)=0$, then clearly the element $w_0\in W$ admits a separating system.
\end{enumerate}
\end{rmk}

\begin{comment}
Suppose $\{s_i \in H^0(\Xcal_{B,w}, \Vscr(\lambda_i))\}_{i=1}^n$ is a separating basis of partial Hasse invariants  for a stratum $\Xcal_{B,w}$. If $s_i$ is not nowhere vanishing, let $\Xcal_{B,w_i}$ be its reduced vanishing locus. For every $s_i$ which is not nowhere vanishing, one has the half-space 
\begin{equation}
C(s_i)^{+}=\{\lambda=\sum_{i=1}^n a_i \lambda_i \in X^*(T) | a_i \geq 0\}
\end{equation}
of vectors whose $i$th coordinate in the basis $\{\lambda_i\}$ is nonnegative. Observe that $\langle C_{B,w} \rangle $ is the intersection of all the $C(s_i)^+$.

\begin{proposition} 
\label{prop-codim-1}
Let $w \in W$. Assume that $\Xcal_{B,w}$ admits a separating basis of Hasse invariants $\{s_i\}_{i=1}^n$ such that, for every $s_i$ which is not nowhere vanishing, one has 
\begin{enumerate}[(a)]
\item 
\label{item-codim-1-cone-equality}
$\langle C_{B, w_i} \rangle = \langle C_{X,B,w_i} \rangle$
\item 
\label{item-siplus}
$\langle C_{B, w_i} \rangle \subset C(s_i)^+$
\end{enumerate}
Then $\langle C_{B, w} \rangle = \langle C_{X,B,w} \rangle$.
\end{proposition}
\end{comment}

\begin{proposition}\label{prop-codim-1}
Let $w\in W$ with lower neighbors $\{w_i\}_{i=1}^n$. Assume that:
\begin{enumerate}[(a)]
\item \label{item-prop1} There exists a separating system of partial Hasse invariants for $\Ycal_{w}$.
\item \label{item-prop2} One has $\bigcap_{i=1}^n  C_{\Ycal, w_i}  \subset  C_{\Ycal, w}$.
\item \label{item-prop3} Each $w_i$ satisfies the equality of cones $C_{Y,w_i}  = C_{\Ycal, w_i}$.
\end{enumerate}
Then $w$ satisfies the equality of cones $ C_{Y,w}  = C_{\Ycal, w}$. 
\end{proposition}

\begin{proof}
Let $\lambda \in C_{Y,w} $ and assume that $\lambda\notin  C_{\Ycal, w}$. Choose a nonzero $f\in H^0(Y^*_{w}, \Lscr_Y(N\lambda))$ for some $N\geq 1$. Let $\{s_i\in H^0(\Ycal^*_{w}, \Lscr_{\Ycal}(\lambda_i))\}_{i=1}^n$ be a separating system of partial Hasse invariants for $\Ycal_{w}$. By \eqref{item-prop2}, there exists $i\in \{1,...,n\}$ such that $\lambda \notin  C_{\Ycal, w_i}$. By~\eqref{item-prop3}, $\lambda \notin C_{Y,w_i}$. Hence $H^0(Y^*_{w_i}, \Lscr_Y(\lambda))=0$. Multiplication by $\zeta_Y^*(s_i)$ gives an exact sequence 
$0\to H^0(Y^*_{w}, \Lscr_Y(N\lambda-\lambda_i)) \to H^0(Y^*_{w}, \Lscr_Y(N\lambda)) \to H^0(Y^*_{w_i}, \Lscr_Y(N\lambda))$. 
Thus $H^0(Y^*_{w}, \Lscr_Y(N\lambda-\lambda_i)) \simeq H^0(Y^*_{w}, \Lscr_Y(N\lambda))$. In particular $N\lambda -\lambda_i \in C_{Y,w}$.

It is clear that $N\lambda - \lambda_i\notin  C_{\Ycal, w}$; otherwise $\lambda$ would also lie in $ C_{w}$. So we may repeat the same procedure to $N\lambda - \lambda_i$, but with possibly a different $i'\in \{1,...,n\}$. Hence there exists a sequence $(i_d)_{d\geq 1}$ with values in $\{1,...,n\}$ such that $N\lambda-\sum_{d=1}^m \lambda_{i_d} \in C_{Y,w}$ for all $m\geq 1$ and such that multiplication by $\prod_{d=1}^m s_{i_d}$ gives an isomorphism 
\begin{equation}
H^0(Y^*_{w}, \Lscr_Y(N\lambda-\sum_{d=1}^m \lambda_{i_d})) \stackrel{\sim}{\to} H^0(Y^*_{w}, \Lscr_Y(N\lambda)).
\end{equation}
There exists $j\in \{1,...,n\}$ such that $i_d=j$ for infinitely many $d\geq 1$. We have shown that $f$ is divisible by $s^m_j$ for all $m\geq 1$. This implies that $s_j$ is nowhere non-vanishing, contradiction.
\end{proof}

\begin{corollary}
\label{cor-codim-1}
Assume that \textnormal{\Prop~\ref{prop-codim-1}\eqref{item-prop1}-\eqref{item-prop2}} hold for all $w \in W$. Then $ C_{Y, w}  = C_{\Ycal,w} $ for all $w \in W$. In particular, $ C_Y =  C_{\Ycal}$.
\end{corollary}

\begin{proposition}\label{colength1}
Let $w\in W$ be a lower neighbor of $w_0$.  Assume that
\begin{enumerate}[(a)]
\item \label{item-col1} The Picard group of $G$ is trivial.
\item \label{item-col2} $X_{+,I}(T)\cap  C_{w}  \subset  C $.
\item \label{item-col3} One has $C_{Y,w}  =  C_{\Ycal, w} $.
\end{enumerate}
Then one has the equality of cones $C_{Y}  = C_{\Ycal} $.
\end{proposition}

\begin{proof}
Let $\lambda \in  C_{Y} $ and assume that $\lambda\notin  C_{\Ycal}$.  Fix a nonzero $f\in H^0(Y,\Lscr_Y(N\lambda))$ for some $N\geq 1$. Since $ C_{Y} \subset X_{+,I}(T)$, we deduce that $\lambda \notin C_{\Ycal, w} =C_{Y,w} $. By \eqref{item-col1}, we can find $\mu\in X^*(T)$ and a partial Hasse invariant $s\in H^0(\Ycal^*_w,\Lscr_{\Ycal}(\mu))$ such that $\Div(s)=\Ycal^*_w$. Since $N\lambda\notin C_{Y,w}$, we have $H^0(Y^*_w,\Lscr_Y(N\lambda))=0$. Hence $f$ restricts to zero along $Y^{*}_w$, and thus $f$ is divisible by $s':=\zeta_{Y}^{*}(s)$ ; there exists $g\in H^0(Y,\Lscr_Y(N\lambda-\mu))$ such that $f=s'g$. We have shown that $N\lambda-\mu\in C_{Y}$, hence $\lambda-\frac{\mu}{N}\in C_Y $.

It is clear that $\lambda-\frac{\mu}{N}\notin  C_{\Ycal} $, because otherwise $\lambda$ would also lie in $C_{\Ycal} $. Repeating this argument, we deduce that $f$ is divisible by $s'^{m}$ for all $m\geq 1$, which is a contradiction, as in the proof of \Prop~\ref{prop-codim-1}.
\end{proof}

\section{Example 1 : Groups of type $A_1^n$} 

\label{sec-Hilbert}
In this section, we study Question~\ref{q-strata-cones} for $\fp$-groups $G$ of type $A_1^n$. As a corollary, we deduce results about Hilbert modular varieties.

We prove Conjecture \ref{conj} when $G$ is of type $A_1^n$ (and $\Zcal$ of Borel-type). Therefore, in all this section one has $X=Y$. If $G$ is $\fp$-split or if $G$ splits over the quadratic extension $\FF_{p^2}$, then we show that $ C_{\Ycal, w}  =  C_{Y,w}$ holds for all $w\in W$, which gives a complete answer to Question \ref{q-strata-cones}.

In the general case, we define a set of admissible strata for which one has $C_{\Ycal,w} = C_{Y,w} $. However, we conjecture that not all $w$ satisfy this equality of cones.

\subsection{Notation}
\label{thegroup}
Let $n\geq 1$ be an integer; let $n=n_1+\cdots +n_r$,  be a partition of $n$ with $n_i\geq 1$ for all $i=1,...,r$. Consider the $\FF_p$-reductive group $G$ defined by
\begin{equation}
G:=G_1\times ...\times G_r, \quad G_i:=\Res_{\FF_{p^{n_i}}/ \FF_p}(SL_{2,\FF_{p^{n_i}}})
\end{equation}
Define $N_m=\sum_{i=1}^{m} n_i$ for all $1\leq m\leq r$ and $N_0:=0$. Denote again by $\sigma$ the permutation of $\{1,...,n\}$ defined as a product $\sigma=c_1\cdots c_r$ where $c_i$ is the $n_i$-cycle $c_i=(N_i \ (N_i-1) \cdots (N_{i-1}+1))$ for $i=1,...,r$. There is an isomorphism
\begin{equation}\label{Gk isom}
G_k\simeq SL_{2,k}^n
\end{equation}
such that the action of $\sigma \in \Gal(k/\FF_p)$ on $G(k)\simeq SL_2(k)^n$ is given by
\begin{equation}
{}^\sigma (x_1,...,x_n):=(\varphi(x_{\sigma(1)}),\varphi(x_{\sigma(2)}),...,\varphi(x_{\sigma(n)})).
\end{equation}

Let $T\subset SL_{2,k}$ be the diagonal torus. We identify $X^*(T)=\ZZ$ by sending $m\in \ZZ$ to the character $\diag(x,x^{-1})\mapsto x^m$.
Define $\widetilde{T}:=T\times...\times T \subset G_k$ and identify similarly $X^*(\widetilde{T})=\ZZ^n$. Let $B\subset SL_{2,k}$ be the Borel subgroup of lower-triangular matrices, and define $\widetilde{B}:=B\times...\times B\subset G$. Denote by $\widetilde{B}_-$ the opposite Borel. The Weyl group of $G$ is $W=S_2\times ... \times S_2$.

Let $\Zcal$ be the Borel-type zip datum $(G,\widetilde{B},\widetilde{T},\widetilde{B}_-,\widetilde{T},\varphi)$. Denote by $\Xcal$ the corresponding stack of $G$-zips. Fix a map $\zeta:X\to \Xcal$ satisfying the assumptions of Conjecture~\ref{conj}. Define a Zariski open subset $U\subset SL_2$ as the non-vanishing locus of the function
\begin{equation}
h:SL_{2,k}\to \AA_k^1, \quad h:\left(\begin{matrix}
a&b\\
c&d
\end{matrix} \right)\mapsto a.
\end{equation}
Denote by $Z\subset SL_{2,k}$ the zero locus of $h$ (note that $Z$ is a reduced subscheme). Identify the elements of $W$ with subsets $S\subset\{1,...,n\}$ by the map
\begin{equation}
W\to \Pcal(\{1,...,n\}), \quad \tau=(\tau_1,...,\tau_n) \mapsto \{i\in \{1,...,n\}: \tau_i=1\}.
\end{equation}
For a subset $S\subset \{1,...,n\}$, write
\begin{equation}\label{Sdisjoint}
S=S_1\sqcup...\sqcup S_r, \quad S_i:=S\cap \{N_{i-1}+1,...,N_i\}.
\end{equation}
The zip stratum corresponding to a subset $S\subset \{1,...,n\}$ is defined by:
\begin{equation}
G_S:=\prod_{i=1}^n G_{S,i}
\end{equation}
where $G_{S,i}:=U$ if $i\in S$ and $G_{S,i}:=Z$ if $i\notin S$. For a subset $S\subset \{1,...,n\}$, denote by $\Xcal_S:=\left[E\backslash G_S\right]\subset \GZip^\Zcal$ and $X_S\subset X$ the corresponding locally closed subsets, endowed with the reduced structure, and define similarly $\Xcal^*_S$ and $X^*_S$ as their respective Zariski closures.

Write $C_S$ and $C_{X,S}$ for the cones corresponding to the zip stratum $G_S$, as defined in section \ref{subsec-cones}. Denote by $e_1,...,e_n\in \ZZ^n$ the natural basis of $\ZZ^n$. For any subset $S\subset \{1,...,n\}$, we define a $\QQ$-basis $\Bcal_S=(\beta_{1,S},...,\beta_{n,S})$ of $\QQ^n$ by:
\begin{equation}
\beta_{i,S}:=\begin{cases}
e_i-pe_{\sigma(i)}& \textrm{ if } i\in S, \\
-e_i-pe_{\sigma(i)} & \textrm{ if } i\notin S.
\end{cases}
\end{equation}
The cone $ C_S$ is the set of characters $\lambda\in X^*(\widetilde{T})$ such that
\begin{equation}\label{lambda dec}
\lambda = \sum_{i\in S} a_i \beta_{i,S}
\end{equation}
where $a_i\in \NN$ for all $i\in S$ and $a_i\in \ZZ$ for all $i\notin S$.

\subsection{The result}
\label{sec-result-hilbert}
Let $d\in \ZZ_{\geq 1}$ be an integer. For any subset $R\subset \ZZ$ consisting of $d$ consecutive integers, the map $\phi_R:R \to \ZZ/d\ZZ$, $k\mapsto \overline{k}$ is a bijection. Let $\ZZ/d\ZZ$ act on itself by addition. Then $\phi_R$ yields a natural action of $\ZZ/d\ZZ$ on the following objects:
\begin{enumerate}[(a)]
\item The set $R$ itself.
\item The powerset $\Pcal(R)$.
\item The set of pairs $(S,j)$ where $S\subset R$ and $j\in S$.
\end{enumerate}

\begin{definition}\label{defadm} \
Let $d\geq1$ be an integer and $R\subset \ZZ$ a subset consisting of $d$ consecutive numbers.
\begin{enumerate}[(1)]
\item A normalized admissible pair of $R$ is a pair $(S,x_r)$ such that $S\subset R$ is of the form $S=\{x_1,...,x_r\}$ with $x_1<...<x_r$ and $x_{i+1}-x_i$ odd for all $1\leq i \leq r-1$.
\item An admissible pair of $R$ is a pair $(S,j)$ that is in the $\ZZ/d\ZZ$-orbit of a normalized admissible pair for $R$.
\item A $G$-admissible pair is a pair $(S,j)$ such that $j\in S_m$ for some $1\leq m\leq r$ (notation as in (\ref{Sdisjoint})) and $(S_m,j)$ is an admissible pair of $\{N_{m-1}+1,...,N_m\}$.
\item A subset $S\subset \{1,...,n\}$ is $G$-admissible if the pair $(S,j)$ is $G$-admissible for all $j\in S$.
\end{enumerate}
\end{definition}

\begin{rmk} The following give examples of $G$-admissible subsets of $\{1,2, \ldots , n\}$:
\begin{enumerate}[(a)] 
\item 
\label{item-adm-singleton}
The singleton $\{i\}$ for all $i\in \{1,...,n\}$;
\item
\label{item-adm-whole-set}
The set $\{1,...,n\}$ itself;
\item
\label{item-adm-frob-trivial}
If $\sigma(i)=i$ for all $i\in \{1,...,n\}$ (equivalently if $r=n$), then every subset $S$ of $\{1,2, \ldots, n\}$ is $G$-admissible.
\end{enumerate}
The cases~\eqref{item-adm-singleton},~\eqref{item-adm-whole-set},~\eqref{item-adm-frob-trivial} correspond respectively to one-dimensional strata, the top-dimensional stratum and the case that $G$ is $\fp$-split.
\end{rmk}

\begin{theorem}\label{mainthm}
Let $\zeta:X\to \Xcal$ be a map satisfying \textnormal{\ref{conj}\eqref{hodgetype}-\eqref{item-assume-loc-const}}. For each $G$-admissible $S\subset \{1,\ldots,n\}$, one has $C_{Y,S} = C_{\Ycal, S}  $. In particular, $C_X  = C_\Xcal$.
\end{theorem}

Theorem \ref{mainthm} will be proved at the end of \S\ref{sec-Hilbert}, as a corollary of Proposition \ref{admthm} below. 
As an application, let $F$ be a totally real number field and let $X$ be the special fiber of a Hilbert modular variety attached to $F$. Using Lemma \ref{change}, we may reduce to Theorem \ref{mainthm}. Hence:

\begin{corollary}\label{cor Hilbert} If $[F:\QQ]>1$, then \textnormal{Conjecture \ref{conj}} holds for $X$. Explicitly:
\begin{equation}
 C_X = C_\Xcal =\left\{ \sum_{i=1}^n a_i(e_i-pe_{\sigma(i)}), \ a_i\in \NN \right\}.
\end{equation}
\end{corollary}
\begin{rmk} When $X$ is a modular curve (so that $F=\QQ$), $X$ fails to satisfy Koecher's principle and Corollary \ref{cor Hilbert} is false, because $X$ is affine and the Hodge line bundle is both ample and anti-ample on $X$. However, Corollary \ref{cor Hilbert} trivially holds for the compactified modular curve $X^{\tor}$.
\end{rmk}
\begin{rmk} When $[F:\QQ]>1$, all EO strata of $X$ of codimension $> 0$ are proper. Hence, in this case, assumption~\ref{conj}\eqref{item-assume-loc-const} already follows from the classical Koecher principle for $X$ (independently of \cite{Lan-Stroh-stratifications-compactifications}).  
\end{rmk}
\subsection{Proof of \textnormal{Theorem \ref{mainthm}}}
Assume $\lambda\in \QQ^n$ is a quasi-character expressed in the basis $\Bcal_S$. Choose an element $j\in S$, and consider the subset $S\setminus \{j\}$, and the corresponding basis $\Bcal_{S\setminus \{j\}}$ of $\QQ^n$. We want to decompose $\lambda$ in the basis $\Bcal_{S\setminus \{j\}}$.

For this, it suffices to determine the decomposition of the vector $\beta_{j,S}=e_j-pe_{\sigma(j)}$ in $\Bcal_{S\setminus \{j\}}$. Write $s_i:=|S_i|$ for $i=1,...r$ and $s:=|S|=\sum_{i=1}^r s_i$. Let $m\in \{1,...,r\}$ such that $j\in S_m$. Let $a_1,...,a_n\in \QQ$ be the unique rational numbers such that
\begin{equation}
\beta_{j,S}=\sum_{i=1}^n a_i \beta_{i,S\setminus \{j\}}.
\end{equation}

One sees immediately that $a_i=0$ for $j\notin \{N_{m-1}+1,...,N_m\}$. For $1\leq a \leq b \leq n$, we define $\gamma(a,b,S) \in \{\pm 1\}$ by the formula:
\begin{equation}
\gamma(a,b,S):=(-1)^{|\{x\in S, \ a\leq x\leq b\}|}.
\end{equation}
For $d\geq 1$, $x=(x_1,...,x_d) \in \QQ^d$ and $y=(y_1,...,y_d)\in \QQ^d$, define $x*y$ as the vector $(x_1y_1,...,x_d y_d)$. Then we have the following formula:

\begin{equation}\label{bigeq}
\delta \left(\begin{matrix}
a_{N_{m-1}+1} \\ a_{N_{m-1}+2} \\ \vdots \\ a_{j-2} \\ a_{j-1} \\ a_j \\ a_{j+1} \\ a_{j+2} \\ \vdots \\ a_{N_m}
\end{matrix} \right)=
\left(\begin{matrix}
(-1)^{s_m+N_m+j} \\ (-1)^{s_m+N_m+j+1} \\ \vdots \\ (-1)^{n_m+s_m-1} \\ (-1)^{n_m+s_m} \\ 1 \\ -1 \\ 1 \\ \vdots \\ (-1)^{N_m+j}
\end{matrix} \right)* \left(\begin{matrix}
\gamma(N_{m-1}+1,j-1,S) \\ \gamma(N_{m-1}+2,j-1,S) \\ \vdots \\ \gamma(j-2,j-1,S) \\ \gamma(j-1,j-1,S) \\ 1 \\ 1 \\ \gamma(j+1,j+1,S) \\ \vdots \\ \gamma(j+1,N_m-1,S)
\end{matrix} \right)* \left(\begin{matrix}
2p^{j-N_{m-1}-1} \\ 2p^{j-N_{m-1}-2} \\ \vdots \\ 2p^{2} \\2 p \\ p^{n_m}+(-1)^{s_m+n_m+1} \\ 2p^{n_m-1} \\ 2p^{n_m-2} \\ \vdots \\ 2p^{j-N_{m-1}}
\end{matrix} \right) \end{equation}
where $\delta=p^{n_m}+(-1)^{s_m+n_m}$.

For the rest of the proof, we abbreviate $H^0(Y^*_S, \Lscr_Y(\lambda)):=H^0(S,\lambda)$.
\begin{lemma}\label{lemmainf}
Let $S\neq \emptyset$ be a subset and $i\in S$. Let $\lambda \in \ZZ^n$ with coordinates $(x_1,...,x_n)$ in $\Bcal_S$. Then there exists $M$ (depending on $S$ and $\lambda$) such that for $m\geq M$, one has
$H^0(S,\lambda-m \beta_{i,S})=0$.
\end{lemma}
\begin{comment}
\begin{proof}
Let $h_i$ be the unique (up to scalar) non-zero section $h_i\in H^0(\Xcal^*_S, \beta_{i,S})$. Multiplication by $h_i$ yields an injection $H^0(X^*_S,\lambda-(m+1)\beta_{i,S}) \to H^0(X^*_S,\lambda-m \beta_{i,S})$ for all $m\geq 0$. Hence the dimension of $H^0(X^*_S,\lambda-m \beta_{i,S})$
this map is an isomorphism for 
\end{proof}
\end{comment}

\begin{proof}
By induction on $s=|S|$. Let $m\in \{1,...,r\}$ such that $i\in S_m$. By~\ref{conj}\eqref{item-assume-loc-const}, the result is clear if $s=1$. So assume $s>1$. Also, we may assume $S_m=\{x_1,...,x_{s_m}\}$ with $x_1<...<x_{s_m}$ and $i=x_1$. Let $j\in S-\{i\}$. By induction, there exists $M(\lambda,j)\geq 1$ such that
$H^0(S-\{j\},\lambda-m \beta_{i,S})=0$
for $m\geq M(\lambda,j)$. 

Consider the unique (up to scalar) non-zero $h_j\in H^0(\Ycal^*_S, \beta_{j,S})$. The vanishing locus of $h_j$ is exactly $\Ycal^*_{S-\{j\}}$, and that of $\zeta^*(h_j)$ is $X^*_{S-\{j\}}$ by smoothness of $\zeta$. Multiplication by $\zeta^*(h_j)$ yields a short exact sequence of sheaves
$$0\to \Lscr_Y(\lambda-m\beta_{i,S}-\beta_{j,S})|_{Y^*_S}\to \Lscr_Y(\lambda-m\beta_{i,S})|_{Y^*_S}\to \Lscr_Y(\lambda-m\beta_{i,S})|_{Y^*_{S-\{j\}}}\to 0$$
and a long exact sequence of cohomology:
$$0\to H^0(S,\lambda-m\beta_{i,S}-\beta_{j,S}) \to  H^0(S,\lambda-m\beta_{i,S})\to H^0(S-\{j\},\lambda-m\beta_{i,S})\to...$$
Hence for $m\geq M(\lambda,j)$, one has an isomorphism
\begin{equation}
H^0(S,\lambda-m\beta_{i,S}-\beta_{j,S}) \simeq  H^0(S,\lambda-m\beta_{i,S})
\end{equation}
Now there exists an integer $M(\lambda-\beta_{j,S},j)\geq M(\lambda,j)$ such that 
\begin{equation}
H^0(S-\{j\},\lambda-\beta_{S,j}-m \beta_{i,S})=0
\end{equation}
for $m\geq M(\lambda-\beta_{j,S},j)$. Applying the exact sequence above for this character shows that
$$H^0(S,\lambda-m\beta_{i,S}-2\beta_{j,S}) \simeq  H^0(S,\lambda-m\beta_{i,S}-\beta_{j,S})\simeq  H^0(X^*_S,\lambda-m\beta_{i,S})$$
for $m\geq M(\lambda-\beta_{j,S},j)$. Continuing this way, it is clear that we can find $M'(\lambda,j)\geq 1$ such that for $m\geq M'(\lambda,j)$, there exists $\lambda'$ with coordinates $(x'_1,...,x'_n)$ in $\Bcal_S$ such that $x'_j<0$ and
$
H^0(S,\lambda-m\beta_{i,S})\simeq H^0(S,\lambda'-m\beta_{i,S})$.
Hence for large $m$, there exists $\mu$ with coordinates $(y_1,...,y_n)$ in $\Bcal_S$ such that $y_j<0$ for all $j\in S$ and $H^0(S,\lambda-m\beta_{i,S})\simeq H^0(S,\mu)$. By~\ref{conj}~\eqref{item-assume-loc-const}, this space is zero.
\end{proof}

\begin{lemma}\label{parity}
Let $d\geq1$ be an integer and $R\subset \ZZ$ a subset consisting of $d$ consecutive numbers. Let $(S,j)$ be a normalized admissible pair for $R$. Write $S=\{\alpha_1,...,\alpha_s\}$ with $s=|S|$ and $\alpha_1<...<\alpha_s$. Then the integer $\alpha_1+\alpha_s+s$ is odd.
\end{lemma}

\begin{proof}
Since $(S,j)$ is a normalized admissible pair for $R$, one has $j=\alpha_s$ and $\alpha_{i+1}-\alpha_i$ is odd for all $i=1,...,s-1$. Hence
\begin{equation}
\alpha_s-\alpha_1=\sum_{i=1}^{s-1}(\alpha_{i+1}-\alpha_i)
\end{equation}
has the same parity as $s-1$.
\end{proof}

\begin{proposition}\label{admthm}
Let $(S,j)$ be a $G$-admissible pair. Let $\lambda\in \ZZ^n$ with coordinates $(x_1,...,x_n)$ in $\Bcal_S$. If $x_j<0$, then $H^0(S,\lambda)=0$.
\end{proposition}

\begin{proof}
We prove the result by induction on $|S|$. Let $m\in \{1,...,r\}$ such that $j\in S_m$.\medskip

\noindent\textbf{1) Reduction to the case when $x_i<0$ for all $i\in S \setminus S_m$.}

Assume therefore $S\neq S_m$ and let $i\in S\setminus S_m$. The pair $(S-\{i\},j)$ is again $G$-admissible. Let $(y_1,...,y_n)$ be the coordinates of $\lambda$ in the basis $\Bcal_{S-\{i\}}$. Since $i\notin S_m$, one has $y_j=x_j<0$, so we deduce by induction that $H^0(S-\{i\},\lambda)=0$. This implies that
$^0(S,\lambda)\simeq H^0(S,\lambda-\beta_{S,i})$.

Applying the same argument successively, one eventually shows that $H^0(S,\lambda)\simeq H^0(S,\lambda')$ for some $\lambda'$ whose coordinates in $\Bcal_S$ are $(x'_1,...,x'_n)$ such that $x'_i<0$ for all $i\in S\setminus S_m$ and $x'_j<0$. Hence we may assume that this holds for $\lambda$ from the start. In particular, if $S_m=\{j\}$, then one already deduces  that $H^0(S,\lambda)=0$ using \ref{conj}~\eqref{item-assume-loc-const}. Therefore, we assume from now on that $s_m>1$ and $x_i<0$ for all $i\in S \setminus S_m$.

Using the Galois action, We may assume that $S_m=\{\alpha_1,...,\alpha_{s_m}\}$, $j=\alpha_{s_m}$ and $\alpha_1<...<\alpha_{s_m}$. The pair $(S-\{\alpha_1\},j)$ is again admissible. Let $(y_1,...,y_n)$ be the coordinates of $\lambda$ in the basis $\Bcal_{S-\{\alpha_1\}}$. The relevant coordinates are $y_{\alpha_2}, ... , y_{\alpha_{s_m}}$. For all $i>1$, one has:
\begin{equation}\label{eq}
y_{\alpha_i}=x_{\alpha_i}+(-1)^{\alpha_i+\alpha_1+i}x_{\alpha_1}\frac{2p^{n_m+\alpha_1-\alpha_i}}{p^{n_m}+(-1)^{s_m+n_m}}
\end{equation}
Take $i=s_m$ (so $\alpha_i=j$) in equation \eqref{eq}. Then Lemma \ref{parity} shows that the integer $\alpha_{s_m}+\alpha_1+s_m$ appearing in the formula above is always odd. Hence the formula for $i=s_m$ reads:
\begin{equation}\label{eqsm}
y_{j}=x_{j}-x_{\alpha_1}\frac{2p^{n_m+\alpha_1-j}}{p^{n_m}+(-1)^{s_m+n_m}}
\end{equation}
This fact will be used in the following. \medskip

\noindent\textbf{2) Reduction to the case when $x_{\alpha_1}<0$.}

If $x_{\alpha_1}\geq 0$ then (\ref{eqsm}) shows that $y_{j}<0$. Since $(S-\{\alpha_1\},j)$ is admissible, we have by induction $H^0(S-\{\alpha_1\},\lambda)=0$. Hence
$
H^0(S,\lambda)=H^0(S,\lambda-\beta_{\alpha_1,S}).$

We can apply this argument to reduce $x_{\alpha_1}$ by one as long as it is non-negative, and the process stops when $x_{\alpha_1}$ reaches the value $-1$.\medskip

\noindent\textbf{3) The case when $n_m+s_m$ even.}

The "jump" from $\alpha_{s_m}$ to $\alpha_1$ has parity $n_m+\alpha_1+\alpha_{s_m}$, which is the same as the parity of $n_m+s_m+1$ by Lemma \ref{parity}. In this case, it is therefore odd. Hence the pair $(S,\alpha_i)$ is also admissible for all $i=1,...,s_m$.

In particular, we may apply the first step to the pair $(S,\alpha_1)$. It then implies that we can reduce to the case when $x_{\alpha_2}$ is negative. By repeating this process for all $\alpha_i$, $i=1,...,s_m$, we obtain $H^0(S,\lambda)=H^0(S,\lambda')$ for some $\lambda'\in \ZZ^n$ with coordinates $(x'_1,...,x'_n)$ in the basis $\Bcal_S$ and $x'_{i}<0$ for all $i\in \{1,...,n\}$. Hence $H^0(S,\lambda)=0$.\medskip

\noindent\textbf{4) The case when $n_m+s_m$ odd.}

In this case, the pair $(S-\{\alpha_1\},\alpha_i)$ is $G$-admissible for all $i=2,...,s_m$ and the pair $(S-\{j\},\alpha_i)$ is $G$-admissible for all $i=1,...,s_m-1$.

Let $(z_1,...,z_n)$ be the change of basis of $\lambda$ to $\Bcal_{S-\{j\}}$. Then
\begin{equation}
z_{\alpha_1}=x_{\alpha_1}+x_{j}\frac{2p^{j-\alpha_1}}{p^{n_m}-1}
\end{equation}
Recall that $x_j$ and $x_{\alpha_1}$ are both negative. Also, recall formula (\ref{eqsm}) above:
\begin{equation}
y_{j}=x_{j}-x_{\alpha_1}\frac{2p^{n_m+\alpha_1-j}}{p^{n_m}-1}
\end{equation}
where $(y_1,...,y_n)$ denote the coordinates of $\lambda$ in the basis $\Bcal_{S-\{\alpha_1\}}$. Note that since we assumed $p>2$, we have
\begin{equation}
\frac{2p^{n_m-1}}{p^{n_m}-1}<1
\end{equation}
because $n_m>1$. Hence one has the following implications
\begin{align}
|x_{\alpha_1}|\leq |x_{j}| &\Longrightarrow y_{j}<0 \\
|x_j|\leq |x_{\alpha_1}| &\Longrightarrow z_{\alpha_1}<0.
\end{align}
We reduce alternately the value of $x_{\alpha_1}$ and $x_{j}$. We may assume for example that $|x_{\alpha_1}|\leq |x_{j}|$. In this case, we have $y_{j}<0$. Applying induction to the admissible subset $(S-\{\alpha_1\},j)$ gives $H^0(S-\{s_1\},\lambda)=0$. This implies that $H^0(S,\lambda)$ does not change (up to isomorphism) when we replace $x_{\alpha_1}$ by $x_{\alpha_1}-1$. We may repeat this argument until $x_{\alpha_1}$ reaches $x_{j}-1$. Then we have $|x_{j}|\leq |x_{\alpha_1}|$, so $z_{\alpha_1}<0$. Analogously we can replace $x_{j}$ by $x_{j}-1$. In this way, we can reduce alternately the values of $x_{\alpha_1}$ and $x_{j}$ to arbitrarily negative integers without changing $H^0(S,\lambda)$ up to isomorphism. Applying Lemma \ref{lemmainf}, we have $H^0(S,\lambda)=0$.
\end{proof}

\begin{proof}[Proof of Theorem \ref{mainthm}]
Fix a $G$-admissible subset $S\subset \{1,...,n\}$. We prove that $ C_{\Ycal,S}  =  C_{Y,S}$. The inclusion $ C_{\Ycal,S} \subset  C_{Y,S} $ is clear. Conversely, let $\lambda \notin C_{\Ycal,S}$. Then there exists $i\in S$ such that $x_i<0$, where $(x_1,...,x_n)$ denote the coordinates of $\lambda$ in the basis $\Bcal_S$. By definition, the pair $(S,i)$ is $G$-admissible. It follows from Proposition \ref{admthm} that $H^0(S,N\lambda)=0$ for all $N\geq 1$. Hence $\lambda \notin  C_{Y,S}$. 
\end{proof}

\section{Further examples and counter-examples: Types $A_2$, $C_2$ and $C_3$}
\label{sec-further}
\subsection{The results}
\label{sec-results-further}
\begin{theorem} 
\label{th-further}
Suppose $\Zcal$ is a proper zip datum of connected-Hodge type whose underlying group $G$ is either of type $C_2$ or $\fp$-split of type $A_2$. Then \textnormal{Conjecture~\ref{conj}} holds for $\Zcal$. More precisely, $C_{Y,w}=C_{\Ycal,w}$ for all $w \in W$, except possibly in each case when $w'$ is the unique lower neighbor of $w_0$ satisfying $w' \notin {}^IW$.
\end{theorem}

For all $w \in W$, the $C_{\Ycal,w}$ are given explicitly in Figures~\ref{fig-A2},\ref{fig-C2}. By the theorem, this gives an explicit description of $C_Y=C_X$. In terms of \S\ref{sec-cn}-\S\ref{sec-a2}, the exceptional strata in Theorem~\ref{th-further} are given respectively by $w=(14)$ and $w=(123)$.  The results recalled in \S\ref{sec-shimura} imply:
\begin{corollary} Let $\gx$ be a Shimura datum with $\GG=GSp(4)$ or $\GG$ a unitary group associated to an imaginary quadratic field in which $p$ splits and $\GG_\RR=GU(2,1)$. 
Let $X$ be the special fiber of the associated Shimura variety at a level hyperspecial at $p$. 
If $X$ satisfies~\textnormal{\ref{conj}\eqref{item-assume-loc-const}}, then~\textnormal{Conjecture~\ref{conj}} holds for $X$. 
\end{corollary}
By contrast, the following gives a counterexample to Conjecture~\ref{conj} when the zip datum is not of connected-Hodge-type.
\begin{proposition}
\label{prop-counterexample}
Let $X$ be the special fiber of the Siegel Shimura variety of type $GSp(6)$ at a level which is hyperspecial at $p$. The inclusion $C_{\Xcal_B} \subset C_Y$ is strict for the pair $(Y, \Psi \circ \zeta_Y)$ \textnormal{(\S\ref{sec-flag-space})}.
\end{proposition}
\subsection{Groups of type $C_n$}
\label{sec-cn}
For $n\geq 1$, let $G$ be an $\fp$-group of type $C_n$ and $\Zcal_\mu=(G,P,L,Q,M,\varphi)$ a zip datum of connected-Hodge-type. Identify the root system of $(T,G)$ with $(\QQ^n, \Phi)$, where \begin{equation} 
\label{eq-roots-sp}
\Phi=\{\pm\ei\pm\ej | 1\leq i \neq j \leq n\} \cup \{\pm 2\ei| 1 \leq i \leq n\}.\end{equation}
Then $W \cong \{\sigma \in S_{2n}|\sigma(a)+\sigma(2n+1-a)=2n+1 \mbox{ for all } 1\leq a \leq 2n\}$.
Fix $\Delta=\{\ee_i-\ee_{i+1}| 1 \leq i \leq n-1\} \cup \{2\ee_n\}$. 

Since $\Zcal_\mu$ is of connected-Hodge-type, $\mu^\ad \in X_*(G^{\ad})$ is minuscule.  The unique $\Delta$-dominant, minuscule cocharacter of $G$ is the fundamental coweight corresponding to $2\ee_n$. Since $L=\cent_G(\mu)$, the type of $L$ is $I=\Delta \setminus \{2\ee_n\}$.  

For $G=Sp(2n)$, we identify $X^*(T) \cong \ZZ^n$ compatibly with~\eqref{eq-roots-sp}.
\subsection{Groups of type $A_2$.}
\label{sec-a2}
Let $G$ be an $\fp$-split group of type $A_2$ and and $\Zcal_\mu=(G,P,L,Q,M,\varphi)$ a zip datum of connected-Hodge-type. Identify the root system of $(T,G)$ with $(\QQ^n, \Phi)$, where $\Phi=\{\pm(\ei-\ej) | 1\leq i < j \leq 3\}$. Then $W \cong S_3$. The two choices $I=\{\ee_1-\ee_2\}$ and $I=\{\ee_2-\ee_3\}$ correspond to isomorphic zip data. Choose $I=\{\ee_1-\ee_2\}$.
For $G=GL(3)$, identify $X^*(T) \cong \ZZ^3$ compatibly with $\Phi$. It suffices to consider characters of the form $(a_1,a_2,0) \in \ZZ^3$; denote such by $(a_1,a_2)$.
\subsection{Cone Diagrams} Recall the sub-cone $C_{\Sbt, w} \subset C_{\Ycal,w}$ for all $w \in W$ (\Def~\ref{def-partial-hasse-schubert}).
In Figures~\ref{fig-A2},\ref{fig-C2},  the equations for $C_{\Sbt, w}$ appear beside each $w \in W$. In Figure~\ref{fig-A2}, $(a_1,a_2)$ stands for $(a_1,a_2,0)$. A line connecting $w$ to a lower element $w'$ means that $w'$ is a lower neighbor of $w$. Furthermore, the line joining $w$ and $w'$ is labeled by  $\lambda \in \ZZ^2$ to indicate that there exists $h\in H^0(\Ycal^*_w, \Lscr_{\Ycal}(\lambda))$ whose vanishing locus is exactly $\Ycal_{w'}$. 

Among the two cases $A_2$-split and $C_2$, every stratum except $\Ycal_{(14)}$ admits a separating system of partial Hasse invariants (\Def~\ref{def-separating}); $\Ycal^*_{(14)}$ has sections $h_1$ and $h_2$ such that $\Div(h_1)=2\Ycal^*_{(1324)}$ and $\Div(h_2)=2\Ycal^*_{(1243)}$.

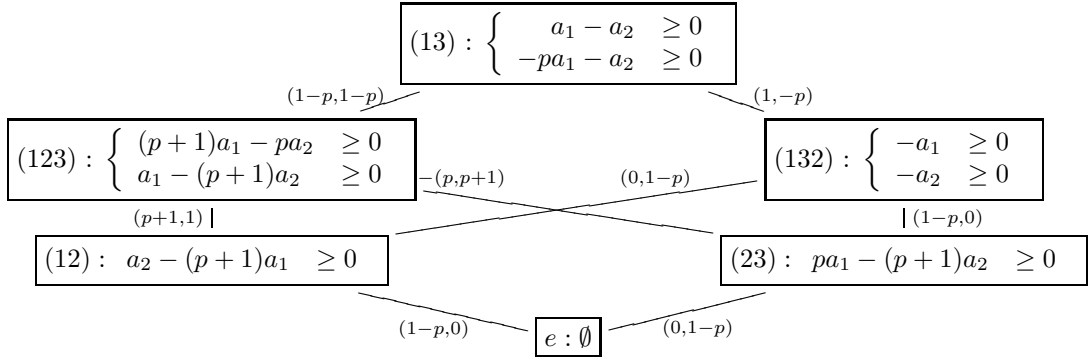
\begin{figure}[h] 
\caption{Strata cones and partial Hasse invariants for type $A_2$-split}
\label{fig-A2}
\centerline{
$\xymatrixcolsep{-1pc}\xymatrixrowsep{.6pc}\xymatrix{  &
\framebox{(13) : $\left\{ \begin{array}{rc} a_1-a_2 & \geq 0 \\ -pa_1 -a_2& \geq 0
\end{array}\right.$ }  
\ar@{-}[rd]^(.6){(1,-p)} 
\ar@{-}[ld]_(.6){(1-p,1-p)}  
&
\\ 
\framebox{(123) : $\left\{ \begin{array}{lc} (p+1)a_1 -pa_2& \geq 0 \\ a_1-(p+1)a_2   & \geq 0
\end{array}\right.$ } 
\ar@{-}[rrd]^(.35){-(p,p+1)}
\ar@{-}[d]_-{(p+1,1)}
&
&
\framebox{(132) : $\left\{ \begin{array}{rc} -a_1& \geq 0 \\ 
-a_2 & \geq 0
\end{array}\right.$ }
\ar@{-}[lld]_(.35){(0,1-p)}
\ar@{-}[d]^-{(1-p,0)}
\\  
\framebox{$(12):  \begin{array}{rc} a_2 -(p+1)a_1 & \geq 0 
\end{array}$
}
\ar@{-}[rd]_-{(1-p,0)}
&
&
\framebox{$(23) :  \begin{array}{rc} pa_1-(p+1)a_2 & \geq 0
\end{array}$
}
\ar@{-}[ld]^-{(0,1-p)}
\\
& 
\framebox{$e :\emptyset$}
&
} $}
\end{figure}
\begin{figure}[h]     
\caption{Strata cones and partial Hasse invariants for type $C_2$}
\label{fig-C2}
 \centerline{
$\xymatrixcolsep{-1pc}\xymatrixrowsep{.65pc}\xymatrix{  &
\framebox{(14)(23) : $\left\{ \begin{array}{rc} a_1-a_2 & \geq 0 \\ -pa_1 -a_2& \geq 0
\end{array}\right.$ }  
\ar@{-}[rd]^(.55){(1,-p)} 
\ar@{-}[ld]_(.55){-(p-1,p-1)}  
&
\\ 
\framebox{(14) : $\left\{ \begin{array}{rc} (p+1)a_1 -(p-1)a_2& \geq 0 \\ -(p-1)a_1-(p+1)a_2 & \geq 0
\end{array}\right.$ } 
\ar@{-}[rrd]
\ar@{-}[d]
& 
&
\framebox{(13)(24) : $\left\{ \begin{array}{rc} -a_2& \geq 0 \\ -a_1 & \geq 0
\end{array}\right.$ } 
\ar@{-}[lld]_(.35){-(0,p-1)}
\ar@{-}[d]^{-(p-1,0)}
\\
\framebox{(1342) : $\left\{ \begin{array}{lc} -(p+1)a_1 -(p-1)a_2& \geq 0 \\ -a_1   & \geq 0
\end{array}\right.$ } 
\ar@{-}[rrd]^(.4){-(0,p+1)}
\ar@{-}[d]_-{(-(p-1), p+1)}
&
&
\framebox{(1243) : $\left\{ \begin{array}{rc} (p-1)a_1 -(p+1)a_2& \geq 0 \\ 
-a_2 & \geq 0
\end{array}\right.$ }
\ar@{-}[lld]_(.4){-(p+1,p-1)}
\ar@{-}[d]^-{-(p+1,0)}
\\  
\framebox{$(12)(34) :  \begin{array}{rc} a_1 -a_2 & \geq 0 
\end{array}$
}
\ar@{-}[rd]_-{-(0,p+1)}
&
&
\framebox{$(23) :  \begin{array}{rc} -pa_1+a_2 & \geq 0
\end{array}$
}
\ar@{-}[ld]^-{(-p,1)}
\\
& 
\framebox{$e :\emptyset$}
&
} $}
\end{figure}
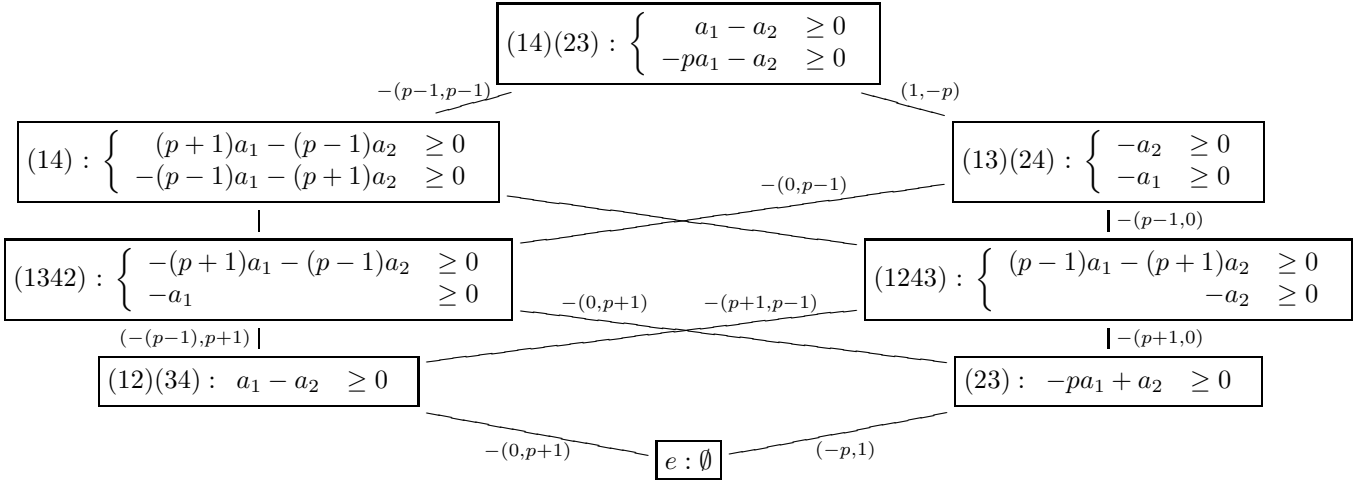

\subsection{Proof of Theorem~\ref{th-further}}
We prove the case $C_2$; the case $A_2$ is completely analogous except it is one step shorter, so it is left to the reader. The inclusions between cones used in the proof below are checked by consulting Figure~\ref{fig-C2}; in case $A_2$ use Figure~\ref{fig-A2}. Note that both $G=Sp(2n)$ and $G=GL(n)$ satisfy $\Pic(G)=0$.
\begin{proof}[Proof of \textnormal{Theorem~\ref{th-further}}] By Lemma~\ref{change}, it suffices to treat $G=Sp(4)$. We proceed in four steps: First, since $(12)(34)$ and $(23)$ have length one, they satisfy the equality of cones by \Prop~\ref{onlyone}. 
Second, since $ C_{(12)(34)} \cap  C_{(23)} \subset C_{(1342)} \cap C_{(1243)}$, both $(1342)$ and $(1243)$ satisfy the equality of cones by \Prop~\ref{prop-codim-1}. Third, as $ C_{(1342)} \cap  C_{(1243)} \subset C_{(13)(24)} $, the equality of cones also holds for $C_{(13)(24)}$, again by \Prop~\ref{prop-codim-1}.

Finally, one has $X_{+,I}^*(T) \cap  C_{(13)(24)} \subset  C_\Xcal $, so the result for $(14)(23)$ follows from the third step and \Prop~\ref{colength1}.
\end{proof}

\subsection{The counter-example}
\label{sec-gsp6}
Let $G=GSp(6)$. Then $\tilde{G}=Sp(6)$ and $\iota:Sp(6) \to GSp(6)$ is the inclusion. Let $X=\Acal_{3, K}$ be the moduli of prime-to-$p$ polarized abelian threefolds in characteristic $p$ with level $K$, assumed hyperspecial at $p$. It is endowed with a smooth morphism $\zeta:X\to \GZip^{\Zcal}$ (\S\ref{sec-shimura}). Let $\omega$ be the Hodge line bundle of $\Acal_{3,K}$. With our identifications, the Hodge character $\eta_{\omega}$ (such that $\Vscr(\eta_{\omega})=\omega$) satisfies $\iota^*(\eta_{\omega})=(-1,-1,-1)$.

Proposition~\ref{prop-counterexample} is an immediate consequence of the following:
\begin{lemma} 
\label{lemma-sp6-example} Let $(\lambda_n)_{n\geq 0} \subset X^*(T)$ be a sequence of characters satisfying \[\iota^*(\lambda_n)=-(p+1+n,p^2+n,p^2+1+n) .\] Then one has: \begin{enumerate}[(a)]
\item \label{item-ex-no-zip-section} $H^0(\Ycal, \Vscr(m\lambda_n))=0$ for all $n,m \geq 1$. In other words, $\lambda_n \not \in  C_\Ycal $ for all $n\geq 1$. 
\item \label{item-ex-ample} For sufficiently large $n$, the line bundle $\Lscr_Y(\lambda_n)$ is ample on $Y$.
\item \label{item-ex-yes-section} In particular, for sufficiently large $n \geq 1$, $\lambda_n\in  C_{Y} $.
\end{enumerate} 
\end{lemma}

\begin{proof}
By Lemma~\ref{change}, part~\eqref{item-ex-no-zip-section} is equivalent to $\iota^*\lambda_n\notin C_{\Ycal}' $ for all $n$, where $C'_{\Ycal}$ is the cone attached to the group $G'$. Using the methods of \cite[\S5]{Goldring-Koskivirta-zip-flags}, we find that the cone $C_{\Ycal}'$ is generated by the three vectors $v_1=(1, 0, -p)$, $v_2=(1, -(p-1), -p)$ and $v_3=-(p-1, p-1, p-1)$, as these are the pullbacks of the three fundamental weights for $G'$ relative to $\Delta$ along the map $\Psi$ (\S\ref{sec-flag-space}). By inverting the $3 \times 3$ matrix whose columns are $v_1,v_2,v_3$, we find that $C'_{\Ycal} $ is the set of tuples $(k_1,k_2,k_3)$ satisfying:
\begin{equation}
\label{eq-sp6-ineq}
\begin{array}{rrll}
  k_1 & -(p+1)k_2 &+k_3  &\geq 0,
\\
 -pk_1 & +(p+1)k_2 & -k_3  &\geq 0, 
\\
  pk_1& &+k_3 & \geq 0. \end{array}
 \end{equation}
\noindent For all $n\geq 1$, $\iota^*\lambda_n$ fails to satisfy the second inequality in~\eqref{eq-sp6-ineq}; hence $\iota^*\lambda_n \notin C'_{\Ycal}$.

Consider~\ref{lemma-sp6-example}\eqref{item-ex-ample}. By Moret-Bailly \cite{Moret-Bailly-abelian-varieties-book}, $\omega$ is ample on $\Acal_{3,K}$. Since $\lambda_0$ is $I$-dominant and regular, it follows from the discussion surrounding Kempf's vanishing theorem \cf \cite[II, \Prop~4.4]{jantzen-representations} that the line bundle $\Lscr_Y(\lambda_0)$ is relatively ample for $Y \to \Acal_{3,K}$. The result now follows from the general lemma which says that, if $f:S_1\to S_2$ is a morphism of schemes, $\Mcal$ is an ample line bundle on $S_2$ and $\Ncal$ is an $f$-ample line bundle on $S_1$, then  $f^*\Mcal^a \otimes \Ncal$ is ample on $S_1$ for sufficiently large $a$ (take $S_1=Y$, $S_2=\Acal_{3,K}$, $\Mcal=\omega$ and $\Ncal=\Vscr(\lambda_0)$).      

Finally,~\eqref{item-ex-yes-section} follows from~\eqref{item-ex-ample} because
every ample line bundle has a positive power which is very ample, whence admits a nonzero global section.
\end{proof}
\subsection{Concluding Remarks}
\label{sec-conclusion}
Conjecture~\ref{conj} concerns equality of the cones $C_Y$ and $C_\Ycal$. 
But we also have the cones $C_{\Sbt} \subset C_\Ycal$ of sections pulled back from $\Sbt$ (\S\ref{sec-flag-space}, \Def~\ref{def-partial-hasse-schubert}). The cone $C_{\Sbt}$ is easily determined for all $G$.

The cases of Conjecture~\ref{conj} proved here all satisfy $C_\Ycal=C_Y=C_{\Sbt}$. Proposition~\ref{prop-counterexample} shows that $C_Y \neq C_{\Sbt}$ when $X=\Acal_{3,K}$. In fact, $C_{\Sbt} \neq C_\Ycal$ in this case; the cone $C_\Ycal$ is much more complicated. This leaves hope that $C_\Ycal=C_Y$ holds even though $C_Y \neq C_{\Sbt}$. 

In conclusion, a first step to prove Conjecture~\ref{conj} for more general groups \eg $G=Sp(2n)$, is to determine the cone $C_\Ycal$. This is the object of our forthcoming work.

\bibliographystyle{plain}
\bibliography{biblio_overleaf}

\begin{thebibliography}{10}

\bibitem{Andreatta-Goren-low-dimension}
F.~Andreatta and E.~Goren.
\newblock Hilbert modular varieties of low dimension.
\newblock In A.~Adolphson, F.~Baldassari, P~Berthelot, N.~Katz, and F.~Loeser,
  editors, {\em Geometric aspects of {D}work theory}, pages 113--175. Walter de
  Gruyter, 2004.

\bibitem{Boxer-thesis}
G.~Boxer.
\newblock {\em Torsion in the coherent cohomology of {S}himura varieties and
  {G}alois representations}.
\newblock PhD thesis, Harvard University, Cambridge, Massachusetts, USA, 2015.

\bibitem{Brunebarbe-Goldring-Koskivirta-Stroh-ampleness}
Y.~Brunebarbe, W.~Goldring, J.-S. Koskivirta, and B.~Stroh.
\newblock Ampleness of automorphic bundles on zip-schemes.
\newblock In preparation.

\bibitem{Diamond-Kassaei}
F.~Diamond and P.~Kassaei.
\newblock Minimal weights of hilbert modular forms in characteristic p.
\newblock arXiv:1612.08725.

\bibitem{Goldring-Koskivirta-Strata-Hasse}
W.~Goldring and J.-S. Koskivirta.
\newblock Strata {H}asse invariants, {H}ecke algebras and {G}alois
  representations.
\newblock Preprint, arXiv:1507.05032.

\bibitem{Goldring-Koskivirta-zip-flags}
W.~Goldring and J.-S. Koskivirta.
\newblock Zip stratifications of flag spaces and functoriality.
\newblock To appear in IMRN, arXiv:1608.01504.

\bibitem{Goren-partial-hasse}
E.~Goren.
\newblock Hasse invariants for {H}ilbert modular varieties.
\newblock {\em Israel J. Math.}, 122:157--174, 2001.

\bibitem{Green-Griffiths-Kerr-Mumford-Tate-Domains-book}
M.~Green, P.~Griffiths, and M.~Kerr.
\newblock {\em Mumford-{T}ate groups and domains: {T}heir Geometry and
  Arithmetic}, volume 183 of {\em Ann. of Math. Studies}.
\newblock Princeton U. Press, Princeton, NJ, 2012.

\bibitem{jantzen-representations}
J.~Jantzen.
\newblock {\em Representations of algebraic groups}, volume 107 of {\em Math.
  Surveys and Monographs}.
\newblock American Mathematical Society, Providence, RI, 2nd edition, 2003.

\bibitem{Kisin-Honda-Tate-theory-Shimura-varieties}
M.~Kisin.
\newblock Honda-{T}ate theory for {S}himura varieties.
\newblock Oberwolfach Report No. 39/2015, ``Reductions of Shimura varieties''.

\bibitem{Koskivirta-compact-hodge}
J.-S. Koskivirta.
\newblock Sections of the {H}odge bundle over {E}kedahl-{O}ort strata of
  {S}himura varieties of {H}odge type, 2014.
\newblock arXiv:1410.1317v2, To appear in J. Algebra.

\bibitem{Koskivirta-Wedhorn-Hasse}
J.-S. Koskivirta and T.~Wedhorn.
\newblock Generalized {H}asse invariants for {S}himura varieties of {H}odge
  type.
\newblock Preprint, arXiv:1406.2178.

\bibitem{Lan-Stroh-stratifications-compactifications}
K.-W. Lan and B.~Stroh.
\newblock Compactifications of subschemes of integral models of {S}himura
  varieties.
\newblock Preprint.

\bibitem{Lee-newton-strata-nonempty}
D.~Lee.
\newblock Non-emptiness of {N}ewton strata of {S}himura varieties of {H}odge
  type.
\newblock Preprint, 2014.

\bibitem{MadapusiHodgeTor}
K.~Madapusi.
\newblock Toroidal compactifications of integral models of {S}himura varieties
  of {H}odge type.
\newblock Preprint.

\bibitem{Moonen-Wedhorn-Discrete-Invariants}
B.~Moonen and T.~Wedhorn.
\newblock Discrete invariants of varieties in positive characteristic.
\newblock {\em IMRN}, 72:3855--3903, 2004.

\bibitem{moonen-gp-schemes}
Ben Moonen.
\newblock Group schemes with additional structures and {W}eyl group cosets.
\newblock In {\em Moduli of abelian varieties ({T}exel {I}sland, 1999)}, volume
  195 of {\em Progr. Math.}, pages 255--298. Birkh\"auser, Basel, 2001.

\bibitem{Moret-Bailly-abelian-varieties-book}
L.~Moret-Bailly.
\newblock {\em Pinceaux de Vari\'et\'es ab\'eliennes}.
\newblock Ast\'erisque. Soc. Math. France, 1985.

\bibitem{Nie-Newton-EO}
S.~Nie.
\newblock Fundamental elements of an affine weyl group.
\newblock {\em Math. Ann.}, 362:485--499, 2015.

\bibitem{Ogus-height-strata-K3}
A.~Ogus.
\newblock Singularities of the height strata in the moduli of {K}3 surfaces.
\newblock In {\em Moduli of abelian varieties}, volume 195, pages 325--343,
  2001.

\bibitem{Oort-stratification-moduli-space-abelian-varieties}
Frans Oort.
\newblock A stratification of a moduli space of abelian varieties.
\newblock In {\em Moduli of abelian varieties ({T}exel {I}sland, 1999)}, volume
  195 of {\em Progr. Math.}, pages 345--416. Birkh\"auser, Basel, 2001.

\bibitem{Pink-Wedhorn-Ziegler-zip-data}
R.~Pink, T.~Wedhorn, and P.~Ziegler.
\newblock Algebraic zip data.
\newblock {\em Doc. Math.}, 16:253--300, 2011.

\bibitem{PinkWedhornZiegler-F-Zips-additional-structure}
R.~Pink, T.~Wedhorn, and P.~Ziegler.
\newblock ${F}$-zips with additional structure.
\newblock {\em Pacific J. Math.}, 274(1):183--236, 2015.

\bibitem{ZhangEOHodge}
C.~Zhang.
\newblock Ekedahl-{O}ort strata for good reductions of {S}himura varieties of
  {H}odge type.
\newblock To appear in Canad. J. Math, arXiv:1312.4869.

\end{thebibliography}
\end{document}